\documentclass{amsart}
\usepackage{amssymb}
\usepackage{verbatim}
\newtheorem{theorem}{Theorem}[section]
\newtheorem{lemma}[theorem]{Lemma}
\newtheorem{proposition}[theorem]{Proposition}

\newtheorem{corollary}[theorem]{Corollary}
\newtheorem{conjecture}[theorem]{Conjecture}

\theoremstyle{definition}
\newtheorem{definition}[theorem]{Definition}

\theoremstyle{remark}
\newtheorem{remark}[theorem]{Remark}

\numberwithin{equation}{section}

\newcommand{\LP}[2]{L^{#1}(X_{#2}, \Sigma_{#2}, m_{#2})}

\newcommand{\M}{\mathcal{M}}

\newcommand{\tM}{\widetilde{\mathcal{M}}}

\newcommand{\I}{1\!{\mathrm l}}
\newcommand{\cM}[2]{{\mathcal M}_{#1}\rtimes_{\sigma^{#2}}{\mathbb R}}

\newcommand{\tr}{{\rm tr}}

\newcommand{\wtr}{\widetilde{\rm tr}}

\begin{document}

\title{A crossed product approach to Orlicz spaces}

\author{Louis Labuschagne}
\address{Internal Box 209, School of Comp., Stat. \& Math. Sci., NWU, Pvt. Bag X6001, 2520 Potchefstroom, South
Africa} 
\email{louis.labuschagne@nwu.ac.za}
\date{\today}
\thanks{This work is based on research supported by the National Research Foundation. Any opinion, findings and conclusions or recommendations expressed in this material, are those of the author, and therefore the NRF do not accept any liability in regard thereto.}

\begin{abstract} 
We show how the known theory of noncommutative Orlicz spaces for semifinite von Neumann algebras equipped with an fns trace, may be recovered using crossed product techniques. Then using this as a template, we construct analogues of such spaces for type III algebras. The constructed spaces naturally dovetail with and closely mimic the behaviour of Haagerup $L^p$-spaces. We then define a modified $K$-method of interpolation which seems to better fit the present context, and give a formal prescription for using this method to define what may be regarded as type III Riesz-Fischer spaces. 
\end{abstract}

\maketitle

\tableofcontents

\section{Preliminaries}

General von Neumann algebraic notation will be based on that of \cite{BRo},  
\cite{Tak} with $\M$ denoting a von Neumann algebra and $\I$ the identity  
element thereof. The projection lattice of a von Neumann algebra 
$\M$ will be denoted by $\mathbb{P}(\M)$. As regards $L_p$-spaces we will use 
\cite{Tp} and \cite{FK} as basic references for the non-commutative context. 
In the case that $\M$ is semifinite, the fns trace of $\M$ will be denoted 
by $\tau_{\M} = \tau$. For such algebras we will denote the space of 
$\tau_\M$-measurable operators affiliated with $\M$ by $\tM$. 

For a von Neumann algebra $\M$ with an \emph{fns} weight $\nu $, the
crossed product of $\M$ with the modular action induced by $\nu$
will be denoted by $\mathcal{A}=\cM{}{}$ and the canonical trace on $\mathcal{A}=\cM{}{}$ by
$\tau_{\mathcal{A}}$. It is well-known that there is a dual action of $\mathbb{R}$ on this crossed product in the form of a one-parameter group of automorphisms $\theta_s$ on 
$\mathcal{A}$ satisfying the condition that $\tau_{\mathcal{A}}\circ\theta_s = e^{-s}\tau_{\mathcal{A}}$. 
The Haagerup $L^p$ spaces $L^p(\M)$ are then realised by means of the canonical extensions 
of these automorphisms to $\widetilde{\mathcal{A}}$. Specifically for $0<p<\infty$ 
we have that $L^p(\M) = \{a\in \widetilde{\mathcal{A}}: \theta_s(a) = e^{-s/p}a \mbox{ for all } s\in\mathbb{R}\}$, with  
$L^\infty(\M) = \{a\in \widetilde{\mathcal{A}}: \theta_s(a) = a \mbox{ for all } s\in\mathbb{R}\}$. Now let $h =
\frac{d\widetilde{\nu}}{d\tau_{\mathcal{A}}}$ where $\widetilde{\nu}$ is
the dual weight of $\nu$ on the crossed product. Then $h$ is a closed densely
defined positive non-singular operator affiliated with the crossed
product. In general, $h$ is not $\tau_{\mathcal{A}}$-measurable, so it has to be treated with care.
We will write $\mathfrak{n}_\nu$ for $\{a\in\M:\nu(a^*a)<\infty\}$, and $\mathfrak{m}_\nu$ 
for the linear span of elements of the form $b^*a$ with $a,b\in\mathfrak{n}_\nu$. 
The $*$-algebra $\mathfrak{m}_\nu$ is of course $\sigma$-weakly dense in $\M$, and is 
linearly spanned by positive elements $a$ from the algebra satisfying
$\nu(a)<\infty$. Although $h$ is in general not pre-measurable, for any 
$a\in \mathfrak{n}_\nu$ the compositions $ah^{1/p}$, 
$h^{1/p}a$ (where $p\geq 2$) will be. To deal with such operators we adopt the 
following convention:
whenever a formula consists of (pre)measurable operators only, their
juxtaposition denotes their strong product; otherwise, it denotes the usual
operator product, and we use square brackets for the closure of a closable
operator. Sometimes we add parentheses to avoid ambiguity. For example, if $h$
is not measurable, but $a,b$ and $h^{1/p}b$ are, we write $a(h^{1/p}b)$ to
denote the strong product of $a$ and $h^{1/p}b$.

By the term an \emph{Orlicz function} we understand a convex function 
$\psi : [0, \infty) \to [0, \infty]$ satisfying $\psi(0) = 0$ and $\lim_{u \to \infty} 
\psi(u) = \infty$, which is neither identically zero nor infinite valued on all of 
$(0, \infty)$, and which is left continuous at $b_\psi = \sup\{u > 0 : \psi(u) < 
\infty\}$. It is worth pointing out that any Orlicz function 
must also be increasing, and continuous on $[0, b_\psi]$. 

Each Orlicz function $\psi$ induces a complementary Orlicz function $\psi^*$ which is defined by 
$\psi^*(u) = \sup_{v > 0}(uv - \psi(v))$. The pair $\psi$ and $\psi^*$ satisfy the following Hausdorff-Young inequality:
$$st \leq \psi(s) + \psi^*(t) \qquad s,t \geq 0.$$
The so-called right-continuous inverse $\psi^{-1}: [0, \infty) 
\to [0, \infty]$ of an Orlicz function is defined by the formula
$$\psi^{-1}(t) = \sup\{s : \psi(s) \leq t\}.$$

Let $\LP{0}{}$ be the space of measurable 
functions on some $\sigma$-finite measure space $(X, \Sigma, m)$. The Orlicz space 
$\LP{\psi}{}$ associated with  $\psi$ is defined to be the set $$L^{\psi} = \{f \in 
L^0 : \psi(\lambda |f|) \in L^1 \quad \mbox{for some} \quad \lambda = \lambda(f) > 0\}.$$
This space turns out to be a linear subspace of $L^0$ which becomes a Banach space when 
equipped with the so-called Luxemburg-Nakano norm 
$$\|f\|_\psi = \inf\{\lambda > 0 : \|\psi(|f|/\lambda)\|_1 \leq 1\}.$$
An equivalent norm (the Orlicz norm) is given by the formula
$$\|f\|^O_\psi = \sup\{|\textstyle{\int_X} fg\, dm| : g\in L^{\psi^*}, \|g\|_\psi\leq1\}.$$
A recently verified fact regarding this norm (see \cite{HM}), is that it may also be realised by the formula
$$\|f\|^O_\psi = \inf_{k > 0}(1 + \|\psi(k|f|)\|_1)/k.$$Although these norms are equivalent, we shall 
hereafter adopt the convention of writing $L^\psi(X, \Sigma, m)$ when the Luxemburg norm is used, and 
$L_\psi(X, \Sigma, m)$ when the Orlicz norm is used. The Lebesgue $L^p$ spaces (with $1\leq p<\infty$) are of course just Orlicz 
spaces corresponding to the Orlicz function $\psi(t)=t^p$. For these spaces we retain the notation 
$L^p$ rather than the more cumbersome $L^{t^p}$.  

Given an Orlicz function $\psi$, in the context of semifinite von Neumann algebras 
$\M$ equipped with an fns trace $\tau$, the process of defining and describing noncommutative versions of the spaces 
$L^\psi$ and $L_\psi$ is fairly well understood. Here the role of 
$L^0$ is played by the $*$-algebra of $\tau$-measurable operators 
$\widetilde{\M}$ (equipped with the topology of convergence in measure). One may either construct these spaces 
directly using brute force, or obtain them as consequences of the very elegant theory of noncommutative 
rearrangement invariant Banach Function Spaces. We briefly review the latter, before making a few comments regarding the former.

Let $\M$ be a semifinite von Neumann algebra. Given an element $f \eta \M$ and $t \in [0, \infty)$, the \emph{generalised singular 
value} $\mu_t(f)$ is defined by $\mu_t(f) = \inf\{s \geq 0 : \tau(\I - \chi_s(|f|)) \leq t\}$ where 
$\chi_s(|f|)$, $s \in \mathbb{R}$, is the spectral resolution of $|f|$. (This directly extends classical notions where for any $f \in \LP{\infty}{}$, 
the function $(0, \infty) \to [0, \infty] : t \to \mu_t(f)$ is known as the decreasing 
rearrangement of $f$.) Since this quantity is in some sense the starting point of the theory of noncommutative Banach Function spaces, we pause review some of the basic properties of this quantity as presented in the paper of Fack and Kosaki \cite{FK}. Having reviewed these, we will freely use them where needed without detailed reference in the ensuing sections.

If we are given $a\in\tM$, the distribution function $[0,\infty]\to[0,\infty]:s\to \lambda_s(a) = \tau(\chi_{(s,\infty)}(|a|))\quad s\geq 0$, exhibits more regular behaviour than if merely $a\eta\M$. Specifically we will then have that $\lambda_s(a)<\infty$ for large $s$ with $\lim_{s\to \infty}\lambda_s(a)=0$. This function then also turns out to be non-increasing and right-continuous. The singular values mentioned above may then also 
be realised in terms of the distribution function by means of the formula $\mu_t(a) = \inf\{s \geq 0 | \lambda_s(a) \leq t\}$ \cite[Proposition 2.2]{FK}. The singular value function $t \to \mu_t(f)$ will generally be denoted by $\mu(f)$. The formula we have just mentioned, shows that $\mu(a)$ is some sort of weak left-inverse of $s\to\lambda_s(a)$. By exploiting this fact, one can show that $\lambda_{\mu_t(a)}(a)\leq t$ for all $t \geq 0$. Further properties of this singular value function include the following \cite[Lemma 2.5]{FK}: (Here we consistently assume $a, b\in \tM$.)
\begin{itemize}
\item The map $(0,\infty)\to [0,\infty]:t\to \mu_t(a)$ is nondecreasing and continuous from the right with $\lim_{t\to 0^+}\mu_t(a)=\|a\|_\infty$ (where $\infty$ is taken as a value for $\|a\|_\infty$ when $a\notin \M$).
\item $\mu_t(a)=\mu_t(|a|)=\mu_t(a^*)$ and $\mu_t(\alpha a)=|\alpha|\mu_t(a)$ for $t>0$.
\item $\mu_{t+s}(a+b)\leq \mu_t(a)+\mu_s(b)$ ($t,s>0$) with in addition $\mu_t(a)\leq \mu_t(b)$ whenever $0\leq a\leq b$. 
\item $\mu_{t+s}(ab)\leq \mu_t(a)\mu_s(b)$ ($s,t>0$) with $\mu_t(ab)\leq \|a\|_\infty\mu_t(b)$ ($t>0$).
\item For any increasing continuous function on $[0,\infty)$ with $f(0)\geq 0$, we have that $f(\mu_t(|a|))=\mu_t(f(|a|))$. 
\end{itemize} 
We further get the following simple but elegant trace formula \cite[Proposition 2.7]{FK}: 
$$\tau(|a|) = \int_0^\infty\mu_t(a)\,dt\quad a\in\tM.$$The topology of convergence in measure itself can also be characterised in terms of the singular value function. Specifically we have that a sequence $\{a_n\}\subset\tM$ converges to $a\in\tM$ in the topology of convergence in measure if and only if for any $t>0$ we have that $\lim_{n\to\infty}\mu_t(a_n-a)=0$ \cite[Lemma 3.1]{FK}. 

In the case of Haagerup $L^p$-spaces (where $\M$ need not be semifinite) we also have a close connection of the topology of these spaces to the singular value function \cite[]{FK}. Specifically for any $a\in L^p(\M)\subset \widetilde{\mathcal{A}}$ ($\infty> p > 0$) we will in this case have that $$\|a\|_p = t^{1/p}\mu_t(a)\quad t>0$$where $\mu(a)$ is relative to the trace $\tau_{\mathcal{A}}$.
   
We proceed to elucidate the use of this function in the theory of noncommutative Banach Function spaces. Classically a function norm 
$\rho$ on $L^0(0, \infty)$ is defined to be a mapping $\rho : L^0_+ \to [0, \infty]$ satisfying
\begin{itemize}
\item $\rho(f) = 0$ iff $f = 0$ a.e.  
\item $\rho(\lambda f) = \lambda\rho(f)$ for all $f \in L^0_+, \lambda > 0$.
\item $\rho(f + g) \leq \rho(f) + \rho(g)$ for all .
\item $f \leq g$ implies $\rho(f) \leq \rho(g)$ for all $f, g \in L^0_+$.
\end{itemize}
Such a $\rho$ may be extended to all of $L^0$ by setting $\rho(f) = \rho(|f|)$, in which case 
we may then define $L^{\rho}(0, \infty)$ to be the space $L^{\rho}(0, \infty) = 
\{f \in L^0(0, \infty) : \rho(f) < \infty\}$. If $L^{\rho}(0, \infty)$ turns out to be a Banach space 
when equipped with the norm $\rho(\cdot)$, we refer to it as a Banach Function Space. 
If $\rho(f) \leq \lim\inf_n \rho(f_n)$ 
whenever $(f_n) \subset L^0$ converges almost everywhere to $f \in L^0$, we say that $\rho$ 
has the Fatou Property. (This is equivalent to the requirement that $\rho(f_n) \uparrow \rho(f)$ whenever $0 \leq f_n \uparrow f$ a.e. \cite[11.4]{AB}.) If less generally this implication only holds for $(f_n) \cup \{f\} 
\subset L^{\rho}$, we say that $\rho$ is lower semi-continuous. If further the situation $f 
\in L^\rho$, $g \in L^0$ and $\mu_t(f) = \mu_t(g)$ for all $t > 0$, forces $g \in L^\rho$ and 
$\rho(g) = \rho(f)$, we call $L^{\rho}$ rearrangement invariant (or symmetric). (The concept of a Banach Function norm can of course in a similar fashion equally well be defined for $L^0(X, \Sigma, \nu)$, where $(X, \Sigma, \nu)$ is an arbitrary measure space.) 

Using the above context, Dodds, Dodds and de Pagter \cite{DDdP} formally defined the noncommutative space 
$L^\rho(\M,\tau)$ to be $$L^\rho(\M,\tau) = \{f \in \widetilde{\M} : \mu(f) \in 
L^{\rho}(0, \infty)\}$$and showed that if $\rho$ is lower semicontinuous and $L^{\rho}(0, 
\infty)$ rearrangement-invariant, $L^\rho(\M,\tau)$ is a Banach space when equipped 
with the norm $\|f\|_\rho = \rho(\mu(f))$. We pause to comment on the three somewhat conflicting notational conventions that have been introduced so far. When considering $L^p$ spaces we will generally use the letters $p$, $q$ or $r$ as indices. When the category of Orlicz spaces is being considered, the relevant Orlicz function will be used as an index. If on the other hand the category of Banach Function Spaces is in view, it is standard practice to index these spaces with the relevant Banach Function norm. In this context we will generally use $\rho$ (with a suitable subscript where appropriate) to denote such a norm. 

For any Orlicz function $\psi$, the Orlicz 
space $L^\psi(0, \infty)$ is known to be a rearrangement invariant Banach Function space 
with the norm having the Fatou Property \cite[Theorem 4.8.9]{BS}. Thus on selecting $\rho$ to be 
one of $\|\cdot\|_\psi$ or $\|\cdot\|^O_\psi$, the above framework  presents us with an elegant means of realising noncommutative 
Orlicz spaces. 

Each rearrangement invariant Banach Function space $X$ over a resonant measure space, admits a so-called \emph{fundamental function} 
$\varphi_X$. (Recall that a $\sigma$-finite measure space is resonant if and only if it is either nonatomic, or else completely atomic with all atoms having equal measure \cite[2.2.7]{BS}.) We make extensive use of these functions for the spaces $L^\rho(\mathbb{R})$ in the ensuing sections, and for this reason briefly review the main features of that theory. For a more comprehensive introduction to this fascinating and useful function, readers are referred to \S 2.5 of \cite{BS}. 

Given a rearrangement invariant Banach Function space $L^\rho(\mathbb{R})$, the fundamental induced by $\rho$ is defined by $\varphi_\rho(t)=\rho(\chi_E)$ where $E$ is a measurable subset of $\mathbb{R}$ for which $\lambda(E)=t$. The fact that this function is well-defined, follows from the rearrangement invariance of the given space. Under some fairly mild restrictions (see (P4) and (P5) of \cite[Definition 1.1.1]{BS}), this function turns out to always be a continuous function on $(0,\infty)$ which is in fact quasi-concave \cite[Corollary 2.5.3]{BS}. Recall that under the term quasi-concave, we understand a non-negative function $\varphi$ on $[0,\infty)$ satisfying the conditions that $$\varphi(t)\mbox{ is increasing on }(0,\infty)\mbox{ and }\frac{\varphi(t)}{t}\mbox{ decreasing, with } \varphi(t)=0 \Leftrightarrow t=0.$$In fact any quasi-concave function on $[0,\infty)$ may conversely be realised as the fundamental function of some rearrangement invariant space \cite[Theorem 2.5.8]{BS}. Although not generally actually concave, any fundamental function is at least equivalent to a concave function in that the least concave majorant $\widetilde{\varphi}$ of a quasi-concave function $\varphi$ satisfies $\frac{1}{2}\widetilde{\varphi}\leq \varphi \leq 2\widetilde{\varphi}$ \cite[Proposition 2.5.10]{BS}. 
This function is well-conditioned to duality \cite[Theorem 2.5.2]{BS} in that for any $X=L^\rho(\mathbb{R})$ with K\"othe 
dual $X'$, we have that $$\varphi_X(t)\varphi_{X'}(t)=t\quad t\geq 0.$$Although it does encode much of the DNA of the inducing space, this function is in general not sharp enough to effectively  distinguish between distinct rearrangement invariant Banach Function spaces \cite[Corollary 2.5.14]{BS}. However for Orlicz spaces the situation is much better \cite[Lemma 4.8.17]{BS}. The fundamental function of the Orlicz space $L^\psi(\mathbb{R})$ equipped with the Luxemburg norm, is nothing but $$\varphi_\psi(t) = \frac{1}{\psi^{-1}(1/t)} \qquad t > 0.$$For the $L^p$-spaces, this boils down to the statement that $\varphi_p(t)=t^{1/p}$ 
($t>0$) \cite[p 65]{BS}. 

We close this introduction by presenting some technical observations regarding the realisation of Orlicz spaces in the context of noncommutative Banach Function spaces, indicating that 
the ``brute force'' approach will yield essentially the same spaces. 
(Proofs may be found in \cite{LM}.)

\begin{lemma}\label{DPvsKlemma} Let $\psi$ be an Orlicz function and $f \in \widetilde{\M}$ a $\tau$-measurable element for which $\psi(|f|)$ is again $\tau$-measurable. Extend $\psi$ to a function on $[0, \infty]$ by setting $\psi(\infty) = \infty$. Then $\psi(\mu_t(f)) = \mu_t(\psi(|f|))$ for any $t \geq 0$. Moreover $\tau_\M(\psi(|f|)) = \int_0^\infty \psi(\mu_t(|f|))\, \mathrm{d}t$.
\end{lemma}

In order to obtain a more elegant reformulation of the results in \cite{LM}, we first make the following technical observation.

\begin{remark}\label{rem:a}
Let $\M$ be a semifinite algebra with \emph{fns} trace $\tau_\M$. The trace $\tau_\M$ may be extended to the extended positive cone ${\widehat{\M}}_+$ of $\M$ \cite[IX.4.9]{Tak}. Any element $a$ 
of ${\widehat{\M}}_+$ is of the form $a=\int_0^\infty \lambda de(\lambda) + \infty.p$ for some spectral resolution $e(\lambda)$ and projection $p$ 
\cite[IX.4.8]{Tak}. For this extension, we have that $\tau_\M(a) = \infty$ if $p\neq0$. So if $\tau_\M(a) < \infty$, we must have that $p=0$, and hence that $a$ is in fact 
an operator affiliated to $\M$. But more is true in this case. For any $n,m \in \mathbb{N}$ with $n < m$, we will then have that $$\tau_\M(e_{(n,m]}) \leq \frac{1}{n}\tau_\M(ae_{(n,m]}) \leq \frac{1}{n}\tau_\M(a) <\infty.$$If now we let $m\to \infty$, we get that $\tau_\M(e_{(n,\infty)}) \leq \frac{1}{n}\tau_\M(a) < \infty$. Thus if $\tau_\M(a) < \infty$, $a$ must in fact correspond to a $\tau_\M$-measurable element of $\tM$ (see \cite[I.21]{Tp}). So given $a \in {\widehat{\M}}_+$, it follows that $\tau_\M(a) < \infty$ if and only if $a$ corresponds to a positive element of $L^1(\M,\tau_\M)$. 

Now given an element $a$ of $\tM$ and a general Orlicz function $\psi$, $\psi(|a|)$ will in general not be a member of $\tM$ (unless of course $\sigma(|a|) \subset [0, b_\psi)$). However we are able to give meaning to $\psi(|a|)$ as an element of ${\widehat{\M}}_+$. If we apply the observation we made previously to this setting, then given $a \in \tM$ it follows that $\tau_\M(\psi(|a|)) < \infty$ if and only if $\psi(|a|)$ corresponds to an element of $L^1(\M,\tau_\M)$.

Now given any $a\eta\M$ and a general Orlicz function $\psi$, $\psi(|a|)$ will in general not be a member of $\tM$ (unless of course $a\in\tM$ and $\sigma(|a|) \subset [0, b_\psi)$). However we are able to give meaning to $\psi(|a|)$ as an element of ${\widehat{\M}}_+$. If we apply the observation we made previously to this setting, then for any $a \eta\M$ it follows that $\tau_\M(\psi(|a|)) < \infty$ if and only if $\psi(|a|)$ corresponds to an element of $L^1(\M,\tau_\M)$. However more is true in this case. Notice that the right-continuous inverse $\psi^{-1}$ is continuous on $[0, \infty)$ with $\psi^{-1}(\psi(t))\geq t$ for all $t \geq 0$ \cite[p 276]{BS}. Thus if indeed $\psi(|a|)\in L^1(\M,\tau_\M)$, then by the functional calculus for positive operators $\psi^{-1}(\psi(|a|))$ will be a $\tau_\M$-measurable element of $\tM$ for which $\psi^{-1}(\psi(|a|))\geq |a|$. This ensures that 
$|a|$, and hence also $a$, is then a $\tau_\M$-measurable element of $\tM$.

To sum up, in terms of the action of $\tau_\M$ on ${\widehat{\M}}_+$, we have that a given $a\eta\M$ will belong to $L^\psi(\M,\tau_\M)$ if and only if $\tau_\M(\psi(\alpha|a|)) < \infty$ for some $\alpha>0$. 
\end{remark}

\begin{proposition}\label{DPvsK} Let $\psi$ be an Orlicz function and let $f \in \widetilde{\M}$ be given. There exists some 
$\alpha > 0$ so that $\int_0^\infty \psi(\alpha\mu_t(|f|))\, \mathrm{d}t < \infty$ if and only if there exists 
$\beta > 0$ so that $\tau_\M(\psi(\beta|f|)) < \infty$. Moreover 
$$\|f\|_\psi=\|\mu(f)\|_\psi = \inf\{\lambda > 0 :  \tau_\M\left(\psi\left(\frac{1}{\lambda}|f|\right)\right) \leq 1\}.$$
\end{proposition}

\begin{proposition}\label{kothe}
Let $\psi$ be an Orlicz function and $\psi^*$ its complementary function. Then $L^{\psi^*}(\M,\tau_\M)$ equipped with the norm $\|\cdot\|^0_{\psi^*}$ defined by $$\|f\|^O_{\psi^*} = \sup \{\tau_\M(|fg|) : g \in L^{\psi}(\M,\tau_\M), \|g\|_\psi \leq 1\} \qquad f \in L^{\psi^*}(\M,\tau_\M)$$is the K\"{o}the dual of $L^{\psi}(\M,\tau_\M)$. That is $$L^{\psi^*}(\M,\tau_\M) = \{f \in \widetilde{\M} : fg \in L^1(\M, \tau_\M) \text{ for all } g \in L^{\psi}(\M,\tau_\M)\}.$$Consequently
$$|\tau_\M(fg)| \leq \|f\|^O_{\psi^*}\cdot\|g\|_\psi \quad\mbox{for all}\quad f \in L^{\psi^*}(\M,\tau_\M), g \in 
L^{\psi}(\M,\tau_\M).$$
\end{proposition}

\section{The semifinite setting}

In this section we will indicate how the theory of Orlicz spaces for semifinite algebras, may be realised using crossed product techniques. So throughout this section we will assume that $\M$ is a semifinite von Neumann algebra with fns trace $\tau_\M$. It is known that in this case the crossed product $\mathcal{A} = \cM{}{}$ will up to Fourier transform correspond to $\M\otimes L^\infty(\mathbb{R})$ (see chapter II of \cite{Tp}). Again up to Fourier transform, the canonical trace on $\mathcal{A}$ is then of the form $\tau_{\mathcal{A}} = \tau_\M \otimes \int_{\mathbb{R}}\cdot e^{-t}dt$ and the Haagerup $L^p$-space $L^p(\M)$, consists of simple tensors of the form $a \otimes e^{\cdot/p}$ where $a\in L^p(\M, \tau_\M)$. Under this identification, the density of the dual weight $\widetilde{\tau_\M}$ with respect to $\tau_{\mathcal{A}}$, is then just $h = \I\otimes e^{t}$. For the sake of clarity of exposition, we shall in this section freely identify $\mathcal{A}=\cM{}{}$ with $\M\otimes L^\infty(\mathbb{R})$, leaving the painful technicalities inherent in verifying that all the properties we verify are preserved by the canonical map identifying these two spaces, to the reader. 

In the introduction we already introduced the convention of writing $L^\psi$ when the Luxemburg norm is in view and $L_\psi$ when the Orlicz norm is used. Given the extensive use we will make of the fundamental functions of these two spaces in this section, we here introduce the additional convention of writing $\varphi_\psi$ for the fundamental function of $L^\psi(\mathbb{R})$ and $\widetilde{\varphi}_\psi$ for the fundamental function of $L_\psi(\mathbb{R})$. These two funtions are actually equivalent to each other in the manner described below.

\begin{remark}
Let $\psi$ be an Orlicz function and $\psi^*$ its adjoint function. From \cite[2.5.2 \& 4.8.17]{BS} we have that $$\varphi_\psi(t) = \frac{1}{\psi^{-1}(1/t)} \qquad\mbox{and}\qquad \widetilde{\varphi}_\psi(t) = \frac{t}{\varphi_{\psi^*(t)}} = t(\psi^*)^{-1}(1/t).$$If we combine these facts with \cite[4.8.16]{BS}, it follows that $$\varphi_\psi(t) \leq \widetilde{\varphi}_\psi(t) \leq 2\varphi_\psi(t) \quad\mbox{for all} \quad t \geq 0.$$
\end{remark}

The fountainhead of the theory is the following adaptation of a lemma by Haagerup. 

\begin{theorem}\label{mainthm}
Let $a \eta \M$ be given. Let $\psi$ be an Orlicz function and let $\varphi_\psi$ be the fundamental function of $L^\psi(\mathbb{R})$ equipped with the Luxemburg norm. Then $$\lambda_\epsilon(a\otimes \varphi_\psi(e^t)) = \tau_\M(\psi(|a|/\epsilon)).$$
\end{theorem}

\begin{proof}
Let $\alpha, \beta > 0$ be given. We proceed to show that $\alpha \psi(\beta) \leq 1 \Leftrightarrow \beta \leq \psi^{-1}(\frac{1}{\alpha})$. It is clear that $$\alpha\psi(\beta) \leq 1 \Leftrightarrow \psi(\beta) \leq \frac{1}{\alpha} \Rightarrow \beta \leq \psi^{-1}(1/\alpha),$$and hence it remains only to show that $\beta \leq \psi^{-1}(\frac{1}{\alpha}) \Rightarrow \alpha \psi(\beta)\leq 1$. Clearly 
$$\beta \leq \psi^{-1}(1/\alpha) \Rightarrow \psi(\beta) \leq \psi(\psi^{-1}(1/\alpha)).$$ Since $\psi(\psi^{-1}(1/\alpha)) \leq 1/\alpha$ (see (8.28) on p 276 of \cite{BS}), we have $\beta \leq \psi^{-1}(\frac{1}{\alpha}) \Rightarrow \alpha\psi(\beta)\leq 1$ as required. It therefore follows that  $$\alpha \psi(\beta) \leq 1 \Leftrightarrow \beta \frac{1}{\psi^{-1}(1/\alpha)} = \beta\varphi_\psi(\alpha) \leq 1,$$or equivalently $$\alpha \psi(\beta) > 1 \Leftrightarrow \beta\varphi_\psi(\alpha) > 1.$$

If we apply this fact to the Borel functional calculus for the commuting positive operators $(\I\otimes e^t)$ and $|a\otimes 1| = |a|\otimes 1$ affiliated to $\M\otimes L^\infty(\mathbb{R})$, we have that 
\begin{eqnarray*}
\chi_{(1,\infty)}(|a\otimes \varphi_\psi(e^t)|) &=& \chi_{(1,\infty)}((|a|\otimes 1)(\I\otimes \varphi_\psi(e^t))) \\
&=& \chi_{(1,\infty)}((|a|\otimes 1)\varphi_\psi(\I\otimes e^t)) \\
&=& \chi_{(1,\infty)}(\psi(|a|\otimes 1)(\I\otimes e^t)) \\
&=& \chi_{(1,\infty)}((\psi(|a|)\otimes 1)(\I\otimes e^t)) \\
&=& \chi_{(1,\infty)}(\psi(|a|)\otimes e^t) 
\end{eqnarray*} 
By Haagerup's lemma (see the discussion preceding this theorem and Lemma II.5 of \cite{Tp}) we have that $$\tau_{\mathcal{A}}(\chi_{(1,\infty)}(|a\otimes 
\varphi_\psi(e^t)|)) = \omega(\I)$$where $\omega$ is the weight $\tau_\M(\psi(|a|)^{1/2}\cdot\psi(|a|)^{1/2})$. (Note that Proposition II.4 of 
\cite{Tp} ensures that the lemma is applicable to $\psi(|a|)\otimes e^t$.) In other words $$\tau_{\mathcal{A}}(\chi_{(1,\infty)}(|a\otimes \varphi_\psi(e^t)|)) =\tau_\M(\psi(|a|)).$$Given $\epsilon > 0$, we therefore have that  
\begin{eqnarray*}
\lambda_\epsilon(a\otimes \varphi_\psi(e^t)) &=& \tau_{\mathcal{A}}(\chi_{(\epsilon,\infty)}(|a\otimes \varphi_\psi(e^t)|))\\
&=& \tau_{\mathcal{A}}(\chi_{(\epsilon,\infty)}(|a|\otimes \varphi_\psi(e^t)))\\
&=& \tau_{\mathcal{A}}(\chi_{(1,\infty)}((|a|/\epsilon)\otimes \varphi_\psi(e^t)))\\
&=& \tau_\M(\psi(|a|/\epsilon))
\end{eqnarray*}
as required.
\end{proof}

The first consequence of this theorem indicates how $L^\psi(\M,\tau_\M)$ may be realised inside $\widetilde{\mathcal{A}}$.

\begin{corollary}[Luxemburg norm]\label{Lux}
Let $\psi$ be an Orlicz function. Given $a\eta\M$, we have that $a\in L^\psi(\M,\tau_\M)$ if and only if $a\otimes\varphi_\psi(e^t)$ belongs to $\widetilde{\mathcal{A}}$. Moreover for any $a\in L^\psi(\M,\tau_\M)$ we have the following formula for the Luxemburg norm: 
$$\|a\|_\psi = \mu_1(a\otimes\varphi_\psi(e^t)).$$
\end{corollary}

\begin{proof}
From Remark \ref{rem:a} it is clear that $a\in L^\psi(\M,\tau_\M)$ if and only if $\tau_\M(\psi(\alpha|a|)) < \infty$ for some $\alpha > 0$. But by the theorem, this is the same as saying that $a\in L^\psi(\M,\tau_\M)$ if and only if $\lambda_{1/\alpha}(a\otimes \varphi_\psi(e^t))< \infty$ for some $\alpha > 0$. On applying \cite[I.21]{Tp}, we have that $a\in L^\psi(\M,\tau_\M)$ if and only if $a\otimes\varphi_\psi(e^t)$ is $\tau_{\mathcal{A}}$-measurable.

Now let $a\in L^\psi(\M,\tau_\M)$ be given. To see the second claim we use the theorem to conclude that
\begin{eqnarray*}
\|a\|_\psi &=& \inf\{\epsilon > 0 | \tau_\M(\psi(|a|/\epsilon)\leq 1\}\\
&=& \inf\{\epsilon > 0 | \lambda_{\epsilon}(a\otimes \varphi_\psi(e^t))\leq 1\}\\
&=& \inf\{\epsilon \geq 0 | \lambda_{\epsilon}(a\otimes \varphi_\psi(e^t))\leq 1\}\\
&=& \mu_1(a\otimes\varphi_\psi(e^t)).
\end{eqnarray*}
(Here the second to last equality follows from the fact that the function $s\to \lambda_s(a\otimes \varphi_\psi(e^t))$ is continuous from the right.)
\end{proof}

In order to realise the space $L_\psi(\M,\tau_\M)$ inside $\widetilde{\mathcal{A}}$, we need the following technical observations:

\begin{lemma}[Orlicz norm]\label{Orllemma}
Let $\psi$ be an Orlicz function. Given $a\eta\M$, we have that $a\otimes\varphi_\psi(e^t)$ belongs to $\widetilde{\mathcal{A}}$ if and only if the same is true of $a\otimes\widetilde{\varphi}_\psi(e^t)$. Moreover for any $a\in L^\psi(\M,\tau_\M)$ and $b\in L_{\psi^*}(\M,\tau_\M)$ the products  $$a\otimes\varphi_\psi(e^t) . b\otimes\widetilde{\varphi}_{\psi^*}(e^t) = ab\otimes e^t$$
$$b\otimes\widetilde{\varphi}_{\psi^*}(e^t).a\otimes\varphi_\psi(e^t) = ba\otimes e^t$$belong to $L^1(\M)$.
\end{lemma}

\begin{proof}
We noted at the start of this section that $\varphi_\psi(t) \leq \widetilde{\varphi}_\psi(t) \leq 2\varphi_\psi(t)$ for all $t \geq 0$. Clearly both $(\varphi_\psi(t)/\widetilde{\varphi}_\psi(t))$ and $(\widetilde{\varphi}_\psi(t)/\varphi_\psi(t))$ are then bounded functions. This in turn ensures that both 
$v = \I \otimes [\widetilde{\varphi}_\psi(e^t)/\varphi_\psi(e^t)]$ and its inverse $v^{-1} = \I\otimes [\varphi_\psi(e^t)/\widetilde{\varphi}_\psi(e^t)]$ belong to $\mathcal{A}$. Clearly $a\otimes\widetilde{\varphi}_\psi(e^t) = v(a\otimes\varphi_\psi(e^t))$ will then be $\tau_{\mathcal{A}}$-measurable whenever $a\otimes\varphi_\psi(e^t)$ is, and vice versa.

The second claim follows directly from the fact that $\varphi_\psi(t)\widetilde{\varphi}_{\psi^*}(t) = t$ for all $t \geq 0$.
\end{proof}

The following result is now a fairly direct consequence of the above proposition considered alongside Corollary \ref{Lux}.

\begin{proposition}\label{Orl}
Let $\psi$ be an Orlicz function. Given $a\eta\M$, we have that $a\in L_\psi(\M,\tau_\M)$ if and only if $a\otimes\widetilde{\varphi}_\psi(e^t)$ belongs to $\widetilde{\mathcal{A}}$. Moreover for any $a\in L_\psi(\M,\tau_\M)$ we have the following formula for the Orlicz norm: 
\begin{eqnarray*}
\|a\|^0_\psi &=& \sup \{\tr(|a\otimes\varphi_\psi(e^t) . b\otimes\widetilde{\varphi}_{\psi^*}(e^t)|) : b \in L^{\psi^*}(\M,\tau_\M), \mu_1(b\otimes\varphi_{\psi^*}(e^t)) \leq 1 \}\\
&=& \sup \{\tr(|b\otimes\widetilde{\varphi}_{\psi^*}(e^t).a\otimes\varphi_\psi(e^t)|) : b \in L^{\psi^*}(\M,\tau_\M), \mu_1(b\otimes\varphi_{\psi^*}(e^t)) \leq 1 \}.
\end{eqnarray*}
\end{proposition}

\begin{proof}
We content ourselves with a few brief comments regarding the final formula. It is an exercise to conclude from the Lemma that for example $\tr(|b\otimes\widetilde{\varphi}_{\psi^*}(e^t).a\otimes\varphi_\psi(e^t)|) = \tau_\M(|ab|)$. In view of the fact that $\|b\|_{\psi^*} = \mu_1(b\otimes\varphi_{\psi^*}(e^t))$, the claim therefore follows from Proposition \ref{kothe}.
\end{proof}

Using the Amemiya form of the Orlicz norm, we obtain the following interesting alternative means of realising $L_\psi(\M,\tau_\M)$ inside $\widetilde{\mathcal{A}}$.

\begin{proposition}
Let $\psi$ be an Orlicz function. The quantity $$|||a\otimes\varphi_\psi(e^t)||| = \inf_{k > 0}k(1+\lambda_k(a\otimes\varphi_\psi(e^t))$$defines a norm for the space $\{a\otimes\varphi_\psi(e^t) | a \in L^\psi(\M,\tau_\M)\}\subset \widetilde{\mathcal{A}}$. Moreover under this norm the canonical map $a\otimes\varphi_\psi(e^t) \to v(a\otimes\varphi_\psi(e^t))$ defines an isometry onto the space $\{a\otimes\widetilde{\varphi}_\psi(e^t) | a \in L_\psi(\M,\tau_\M)\}\subset \widetilde{\mathcal{A}}$ equipped with the norm introduced in Corollary \ref{Orl}. (Here $v$ is as in the proof of Lemma 
\ref{Orllemma}.)
\end{proposition}

\begin{proof}
To prove this we first apply Theorem \ref{mainthm} to see that
\begin{eqnarray*}
|||a\otimes\varphi_\psi(e^t)||| &=& \inf_{k > 0}k(1+\lambda_k(a\otimes\varphi_\psi(e^t))\\
&=& \inf_{k > 0}\frac{1}{k}(1+\lambda_{1/k}(a\otimes\varphi_\psi(e^t))\\
&=& \inf_{k > 0}\frac{1}{k}(1+\tau_\M(\psi(k|a|))).
\end{eqnarray*} 
This last expression is precisely the formula for the Orlicz norm on $L^\psi(\M,\tau_\M)$ in Amemiya form. Hence the claim follows. 
\end{proof}

In the remainder of this section we will where confusion is possible, adopt the convention of writing $a, b, c,\dots$ when operators affiliated to $\M$ are in view, and $\widetilde{a}, \widetilde{b}, \widetilde{c},\dots$  for operators affiliated to $\mathcal{A}$. 

\begin{definition}
On the basis of Corollary \ref{Lux} and Proposition \ref{Orl}, we will denote the space $\{a\otimes\varphi_\psi(e^t) | a \in L^\psi(\M,\tau_\M)\}\subset \widetilde{\mathcal{A}}$ equipped with the norm $\|a\otimes\varphi_\psi(e^t)\|_\psi = \mu_1(a\otimes\varphi_\psi(e^t))$ by $L^\psi(\M)$. With reference to this space, we will denote the space 
$\{a\otimes\widetilde{\varphi}_\psi(e^t) | a \in L_\psi(\M,\tau_\M)\}\subset \widetilde{\mathcal{A}}$ equipped with the norm
\begin{eqnarray*}
\|a\otimes\widetilde{\varphi}_\psi(e^t)\|^O_\psi &=& \sup \{\tr(|\widetilde{b}.[a\otimes\widetilde{\varphi}_{\psi^*}(e^t)]|) : \widetilde{b} \in L^{\psi^*}(\M), \|\widetilde{b}\|_{\psi^*} \leq 1 \}\\
&=& \sup \{\tr(|[a\otimes\widetilde{\varphi}_{\psi^*}(e^t)].\widetilde{b}|) : \widetilde{b} \in L^{\psi^*}(\M), \|\widetilde{b}\|_{\psi^*} \leq 1 \}.
\end{eqnarray*}
by $L_\psi(\M)$.
\end{definition}

\begin{remark}
As we saw earlier, $\varphi_\psi(t) \leq \widetilde{\varphi}_\psi(t) \leq 2\varphi_\psi(t)$ for all $t \geq 0$.  For any $a\in L^\psi(\M,\tau_\M)$ we therefore have that $|a|\otimes \varphi_\psi(e^t) \leq |a|\otimes\widetilde{\varphi}_\psi(e^t) \leq 2|a|\otimes\varphi_\psi(e^t)$, and hence that $$\mu_s(a\otimes \varphi_\psi(e^t)) \leq \mu_s(a\otimes\widetilde{\varphi}_\psi(e^t)) \leq 2\mu_s(a\otimes\varphi_\psi(e^t)) \quad \mbox{for all} \quad s \geq 0.$$  Although unable to verify this yet, the author suspects that $\mu_1(a\otimes\widetilde{\varphi}_\psi(e^t)) = \|a\otimes\widetilde{\varphi}_\psi(e^t)\|^O_\psi$. If this does indeed prove to be the true, then for the case $s=1$ the above inequality would then translate to nothing but the classical equivalence of the Orlicz and Luxemburg norms \cite[4.8.14]{BS}: namely that $\|a\|_\psi \leq \|a\|^O_\psi \leq 2\|a\|_\psi$ for all $a\in L^\psi(\M,\tau_\M)$.
\end{remark}

Having introduced ``crossed product versions'' of $L^\psi$ and $L_\psi$, we pass to the issue of K\"othe duality.

\begin{definition}
We define the K\"othe dual of $L^\psi(\M)$ to be the space $$\{\widetilde{a} \in \widetilde{\mathcal{A}}| \tilde{a}\tilde{b}, \tilde{b} \tilde{a} \in L^1(\M) \mbox{ for all } \widetilde{b} \in L^{\psi}(\M)\}$$equipped with the norm 
\begin{eqnarray*}
\|\widetilde{a}\| &=& \sup \{\tr(|\widetilde{b}.\widetilde{a}|) : \widetilde{b} \in L^{\psi}(\M), \|\widetilde{b}\|_\psi \leq 1 \}\\
&=& \sup \{\tr(|\widetilde{a}.\widetilde{b}|) : \widetilde{b} \in L^{\psi}(\M), \|\widetilde{b}\|_\psi \leq 1 \}.
\end{eqnarray*}
\end{definition}

\begin{theorem}\label{kothecross}
For any $\widetilde{a} \in \widetilde{\mathcal{A}}$ the following are equivalent:
\begin{enumerate}
\item $\widetilde{a} \in L_{\psi^*}(\M)$
\item $\widetilde{a}\widetilde{b} \in L^1(\M)$ for all $\widetilde{b} \in L^{\psi}(\M)$
\item $\widetilde{b}\widetilde{a} \in L^1(\M)$ for all $\widetilde{b} \in L^{\psi}(\M)$
\end{enumerate}
In particular $L_{\psi^*}(\M)$ is the K\"othe dual of $L^\psi(\M)$.  
\end{theorem}

\begin{proof} Notice that the final statement is a trivial consequence of the equivalence stated in the first part of the theorem. We prove the equivalence of the first two conditions. (The proof that conditions 1 and 3 are equivalent is similar.) The converse implication being clear, we prove that condition 2 implies that $\widetilde{a} \in L_{\psi^*}(\M)$. So let $\widetilde{a} \in \widetilde{\mathcal{A}}$ be given with $\widetilde{a}\widetilde{b} \in L^1(\M)$ for all $\widetilde{b} \in L^\psi(\M)$. The result will follow from the known definitions if we are able to show that this condition suffices to ensure that $\widetilde{a} \in L_{\psi^*}(\M)$. In other words we need to show that $\widetilde{a}$ can be written as a simple tensor of the form $a\otimes \widetilde{\varphi}_\psi(e^t)$ for some $a \in L_{\psi^*}(\M,\tau_\M)$.

Any projection $e \in \M$ with $\tau_\M(e) < \infty$ will belong to $L^\psi(\M,\tau_\M)$, and hence $e\otimes \varphi_\psi(e^t)$ will belong to $L^\psi(\M)$. So for any such projection we will have that $\widetilde{a}.(e\otimes\varphi_\psi(e^t)) \in L^1(\M)$. That is $\widetilde{a}.(e\otimes\varphi_\psi(e^t)) = g_e \otimes e^t$ for some $g_e \in L^1(\M,\tau_\M)$ with $s_R(g_e) \leq e$, where $s_R(g_e)$ denotes the right support of $g_e$. Equivalently this means that $\widetilde{a}.(e\otimes 1)$ is of the form $g_e \otimes \frac{e^t}{\varphi_\psi(e^t)} = g_e \otimes \widetilde{\varphi}_{\psi^*}(e^t)$ for some $g_e \in L^1(\M,\tau_\M) \subset \tM$ with $s_R(g_e) \leq e$. Here we used the fact that $\varphi_\psi(t).\widetilde{\varphi}_{\psi^*}(t) = t$ for all $t \geq 0$. If we pick a mutually orthogonal family $\{e_\lambda\}$ of such projections for which $\sum_\lambda e_\lambda = \I$, then this means that $\widetilde{a}$ is of the form $\widetilde{a} = \sum_{\lambda} \widetilde{a}(e_\lambda \otimes 1) = \sum_{\lambda}g_\lambda\otimes \widetilde{\varphi}_{\psi^*}(e^t) = g\otimes \widetilde{\varphi}_{\psi^*}(e^t)$ where $ g = \sum_\lambda g_{e_\lambda}$. An application of Proposition \ref{Orl} noe reveals that $g$ is a $\tau_\M$-measurable element of $\tM$, and hence that $\widetilde{a}$ is of the required form. 
\end{proof}

With a little more effort an essentially similar argument can be used to show that $L^{\psi^*}(\M)$ is the K\"othe dual of $L_\psi(\M)$.

The techniques demonstrated in the above proof, allow us to formulate an elegant test for membership of $L^\psi$ and $L_\psi$.

\begin{remark}\label{rem:b}
Let $\widetilde{a} \in \widetilde{\mathcal{A}}$ be given. In proving the above result on K\"othe duality, the only fact we needed to be able to conclude that $\widetilde{a} \in L_{\psi^*}(\M)$, was that $\widetilde{a}.(e\otimes \varphi_\psi(e^t)) \in L^1(\M)$ for any projection $e$ in $\M$ with $\tau_\M(e) < \infty$. One could similarly have shown that requiring that either $$(e\otimes \varphi_\psi(e^t)).\widetilde{a} \in L^1(\M)  \mbox{ for all } e \in \mathbb{P}(\M)\mbox{ with } \tau_\M(e) < \infty,$$or that $$(e\otimes \sqrt{\varphi_\psi(e^t)}).\widetilde{a}.(e\otimes \sqrt{\varphi_\psi(e^t)}) \in L^1(\M) \mbox{ for all } e \in \mathbb{P}(\M)\mbox{ with }\quad \tau_\M(e) < \infty,$$would have sufficed to ensure that $\widetilde{a} \in L_{\psi^*}(\M)$. 

The claim regarding the second alternative condition is maybe less obvious, so we will pause to justify it. Observe that if $(e\otimes \sqrt{\varphi_\psi(e^t)}).\widetilde{a}.(e\otimes \sqrt{\varphi_\psi(e^t)}) \in L^1(\M)$ for any projection $e \in \M$ with $\tau_\M(e) < \infty$, then also $(e\otimes \sqrt{\varphi_\psi(e^t)}).\widetilde{a}.(f\otimes \sqrt{\varphi_\psi(e^t)}) \in L^1(\M)$ for any two projections $e,f \in \M$ 
with finite trace. To see this simply multiply $((e\vee f)\otimes \sqrt{\varphi_\psi(e^t)}).\widetilde{a}.((e\vee f)\otimes \sqrt{\varphi_\psi(e^t)})$ left and right by $e\otimes 1$ and $f \otimes 1$ respectively. If as in the proof of Theorem \ref{kothecross} we pick a 
mutually orthogonal family $\{e_\lambda\}$ of such projections for which $\sum_\lambda e_\lambda = \I$, then this means that each $(e_\lambda\otimes \sqrt{\varphi_\psi(e^t)}).\widetilde{a}.(e_\mu \otimes \sqrt{\varphi_\psi(e^t)})$ is of the form $g_{\lambda \mu} \otimes e^t$ for some $g_{\lambda \mu} \in L^1(\M,\tau_\M)$ with $s_R(g_{\lambda \mu}) \leq e_\lambda$, $s_L(g_{\lambda \mu}) \leq e_\mu$. (As before $s_R(g_{\lambda \mu})$ denotes the right support of $g_{\lambda \mu}$, and $s_L(g_{\lambda \mu})$ the left support.) An adaptation of the argument in the proof of 
that theorem then shows that formally $\widetilde{a} = g \otimes \widetilde{\varphi}_{\psi^*}(e^t)$ where $g = \sum_{\lambda, \mu} g_{\lambda \mu}$. The rest of the proof now proceeds as before.

Now notice that $e\otimes \varphi_\psi(e^t)$ and $(e\otimes \sqrt{\varphi_\psi(e^t)})$ can respectively be written as $(e\otimes 1)\varphi_\psi(h)$ and $(e\otimes 1)\varphi_\psi^{1/2}(h)$ where $h = \I\otimes e^{t}$ is the density of the dual weight $\widetilde{\tau_\M}$ with respect to $\tau_{\mathcal{A}}$. If we momentarily adopt the popular convention of writing $e$ for $ e\otimes 1$, the criteria for membership of $L_{\psi^*}(\M)$ is then one of the following three equivalent conditions:
\begin{itemize}
\item $e\varphi_\psi(h)\widetilde{a} \in L^1(\M)$ for every projection $e \in \M$ with $\tau_\M(e) < \infty$.
\item $\widetilde{a}\varphi_\psi(h)e \in L^1(\M)$ for every projection $e \in \M$ with $\tau_\M(e) < \infty$.
\item $e\varphi_\psi^{1/2}(h)\widetilde{a}\varphi_\psi^{1/2}(h)e \in L^1(\M)$ for every projection $e \in \M$ with $\tau_\M(e) < \infty$.
\end{itemize}
Swopping the roles of $\psi$ and $\psi^*$, then gives the criteria for membership of $L_{\psi}(\M)$. The above argument can now be suitably modified to yield the following criteria for membership of $L^\psi(\M)$: 
\begin{eqnarray*}
&&\widetilde{a} \in  L^\psi(\M)\\
&\Leftrightarrow& e\widetilde{\varphi}_{\psi^*}(h)\widetilde{a} \in L^1(\M) \mbox{ for every projection } e \in \M \mbox{ with } \tau_\M(e) < \infty\\
&\Leftrightarrow& \widetilde{a}\widetilde{\varphi}_{\psi^*}(h)e \in L^1(\M) \mbox{ for every projection } e \in \M \mbox{ with } \tau_\M(e) < \infty\\ 
&\Leftrightarrow& e\widetilde{\varphi}_{\psi^*}^{1/2}(h)\widetilde{a}\widetilde{\varphi}_{\psi^*}^{1/2}(h)e \in L^1(\M) \mbox{ for every projection } e \in \M \mbox{ with } \tau_\M(e) < \infty.
\end{eqnarray*} 
\end{remark}

\begin{theorem}\label{thm:normtop}
The norm topology on $L^\psi(\M)$ (respectively $L_\psi(\M)$) is complete and is homeomorphic to the relative topology induced by the topology of convergence in measure on $\widetilde{\mathcal{A}}$. 
\end{theorem}

Note that this result also holds for $L^\infty$ -- a fact which seems to be new information!

\begin{proof}
We remind the reader that the basic neighbourhoods of zero for the topology of convergence in measure on $\widetilde{\mathcal{A}}$ are of the form $$\mathcal{N}(\epsilon, \delta)=\{\widetilde{a} | \lambda_\epsilon(\widetilde{a}) \leq \delta\}.$$Details may be found in \cite{Tp}.

\noindent\textbf{Case 1:} We first consider the space $L^\psi(\M)$. Let $1> \epsilon > 0$ be given and  suppose that $\|\widetilde{a}\|_\psi = \mu_1(\widetilde{a})< \epsilon$ for some $\widetilde{a} = a \otimes \varphi_\psi(e^t) \in L^\psi(\M)$. But then there must exist an $0 < \alpha < \epsilon$ so that $\lambda_\alpha(a \otimes \varphi_\psi(e^t)) \leq 1$. By Theorem \ref{mainthm}, this ensures that $\tau_\M(\psi(|a|/\alpha)) \leq  1$. Next notice that $\frac{1}{\sqrt{\alpha}} \geq 1$, since by assumption $0< \alpha < \epsilon < 1$. We may therefore use the convexity of $\psi$ to conclude that $\frac{1}{\sqrt{\alpha}}\tau_\M(\psi(|a|/\sqrt{\alpha})) \leq \tau_\M(\psi(|a|/\alpha)) \leq  1$, in other words $\tau_\M(\psi(|a|/\sqrt{\alpha})) \leq \sqrt{\alpha}$. Once again using Theorem \ref{mainthm}, this can be reformulated as $\lambda_{\sqrt{\alpha}}(a \otimes \varphi_\psi(e^t)) \leq \sqrt{\alpha}$. This ensures that $\widetilde{a} = a \otimes \varphi_\psi(e^t) \in \mathcal{N}(\sqrt{\alpha}, \sqrt{\alpha})$ \cite[I.7]{Tp}.

Conversely suppose that for some $\epsilon, \delta > 0$ with $\delta \leq 1$ we have that $\widetilde{a} = a \otimes \varphi_\psi(e^t) \in \mathcal{N}(\epsilon, \delta)$. Then by \cite[I.7]{Tp}, $\lambda_\epsilon(a \otimes \varphi_\psi(e^t)) \leq \delta \leq 1$. It then follows that $\|a \otimes \varphi_\psi(e^t)\|_\psi = \mu_1(a \otimes \varphi_\psi(e^t)) \leq \epsilon$.

The norm topology must therefore be homeomorphic to the relative topology induced by the topology of convergence in measure on $\widetilde{\mathcal{A}}$. 

It remains to verify the claim regarding completeness. This however follows from the completeness of the space $L^\psi(\M,\tau_\M)$, and the fact that $a \to a \otimes \varphi_\psi(e^t)$ defines an isometry from $L^\psi(\M,\tau_\M)$ to $L^\psi(\M)$ (see Corollary \ref{Lux} and the definition of the norm on $L^\psi(\M)$).
 
\noindent\textbf{Case 2:} We now pass to the case of the space $L_\psi(\M)$. This case follows directly from the previous one by an application of the map defined in the proof of Lemma \ref{Orllemma}. In that proof we introduced the operators $v = \I \otimes [\widetilde{\varphi}_\psi(e^t)/\varphi_\psi(e^t)]$ and $v^{-1} = \I\otimes [\varphi_\psi(e^t)/\widetilde{\varphi}_\psi(e^t)]$, and noted that the map $\widetilde{a} \to v\widetilde{a}$ maps $L^\psi(\M)$ onto $L_\psi(\M)$.  (Expressed in terms of $h = \I \otimes e^t$, these operators are just $v= \widetilde{\varphi}_\psi(h)[\varphi_\psi(h)]^{-1}$ and $v^{-1} = \varphi_\psi(h)[\widetilde{\varphi}_\psi(h)]^{-1}$.) Since both $v$ and $v^{-1}$ belong to $\mathcal{A} =\cM{}{}$ (as was noted in the proof of Lemma \ref{Orllemma}), the map $\widetilde{a} \to v\widetilde{a}$ is a homeomorphism with respect to the topology of convergence in measure. So the relative topologies induced by the topology of convergence in measure on these two spaces must be equivalent. To see that the norm topologies are also equivalent, we may combine the known equivalence of the norms on $L^\psi(\M,\tau_\M)$ and $L_\psi(\M,\tau_\M)$ with the isometries defined in Corollaries \ref{Lux} and \ref{Orl}, to conclude that the induced canonical map $L^\psi(\M) \to L^\psi(\M,\tau_\M) \to L_\psi(\M,\tau_\M) \to L_\psi(\M)$ is a norm-isomorphism. 
\end{proof}

Thus far we have seen that the simple prescription of mapping $a$ to $a\otimes \varphi_\psi(e^t)$, provides a very effective method of representing the Orlicz space $L^\psi(\M,\tau_\M)$ inside $\widetilde{\mathcal{A}}$. But what about a general rearrangement invariant Banach Function space $L^\rho(\M,\tau_\M)$? If we follow the same prescription, will $\widetilde{\mathcal{A}}$ be big enough to accommodate a copy of $L^\rho(\M,\tau_\M)$? Exactly how much room is there for such spaces inside $\widetilde{\mathcal{A}}$? To answer this question we need the following sharpening of 
\cite[2.5.10]{BS}.

\begin{theorem}
Let $\varphi_\rho$ be the fundamental function of the rearrangement invariant Banach Function space $L^\rho(\mathbb{R})$. Then there exists a constant $k > 0$ and a function $\varphi_0 : [0,\infty) \to [0,\infty)$ with the properties 
\begin{itemize}
\item $\varphi_0$ is concave and increasing;
\item $\varphi_0$ is continuous on $(0,\infty)$;
\item $\varphi_0(t_0) = 0$ if and only if $t_0 = 0$, and $\sup\{t: \varphi_0(t) < \lim_{t\to\infty} \varphi_0(t)\} = \infty$;
\item $t\to\frac{\varphi_0(t)}{t}$ is decreasing on $(0, \infty)$;
\end{itemize} 
for which we have that $$\frac{1}{k}\varphi_0 \leq \varphi_\rho \leq {k}\varphi_0.$$ 
\end{theorem}

\begin{proof}
In general $\varphi_\rho$ is quasi-concave. However there does exist a concave fundamental function $\widetilde{\varphi}$ satisfying $\frac{1}{2}\widetilde{\varphi} \leq \varphi_\rho \leq 2\widetilde{\varphi}$. Now let $d=\lim_{t\to\infty}\widetilde{\varphi}(t)$ and $c=\sup\{t: \widetilde{\varphi}(t) < d\}$. We proceed to show how $\widetilde{\varphi}$ may be modified to produce the required function $\varphi_0$.  If $d=\infty$, then of course $c=\infty$. Hence in this case there is nothing 
to do, since we may simply set $\varphi_0 = \widetilde{\varphi}$. Now suppose that $c < \infty$. Let $0 < \epsilon < c$ be given and consider the function $$g(t) = \left\{\begin{array}{ll} t & \qquad 0 \leq t \leq \epsilon \\ \frac{(t-\epsilon)(c-\epsilon)}{(t-\epsilon)+(c-\epsilon)} + \epsilon & \qquad t > \epsilon \end{array} \right. .$$Now let $\varphi_0 = \widetilde{\varphi}\circ g$. To see that $\varphi_0$ has the required properties note that
\begin{itemize}
\item $\varphi_0$ must be concave since both $g$ and $\widetilde{\varphi}$ are concave.
\item Since both $g$ and $\widetilde{\varphi}$ are increasing, $\varphi_0$ is also increasing;
\item By construction $\varphi_0(t_0) = 0 \Leftrightarrow g(t_0) = 0 \Leftrightarrow t_0 = 0$;
\item Since $g$ is increasing and $\frac{\widetilde{\varphi}(t)}{t}$ decreasing, $\frac{\varphi_0(t)}{g(t)}$ is decreasing. Since by construction $\frac{g(t)}{t}$ is decreasing, we therefore must have that $\frac{\varphi_0(t)}{t} = \frac{\varphi_0(t)}{g(t)}\frac{g(t)}{t}$ is decreasing.
\end{itemize}
It remains to show that there exists a constant $k > 0$ so that $\frac{1}{k}\varphi_0 \leq \varphi_\rho \leq {k}\varphi_0$. In view of the fact that  
$\frac{1}{2}\widetilde{\varphi} \leq \varphi_\rho \leq 2\widetilde{\varphi}$, this will follow if we can demonstrate the existence of a constant $m > 0$ so that $\frac{1}{m}\varphi_0 \leq \widetilde{\varphi} \leq {m}\varphi_0$. Since $\widetilde{\varphi}$ is increasing and  $g(t) \leq t$, it is immediately obvious that $\varphi_0 = \widetilde{\varphi}\circ g \leq \widetilde{\varphi}$. But by construction $\frac{\widetilde{\varphi}}{\varphi_0}$will then be a continuous function on  $(0, \infty)$ for which $\frac{\widetilde{\varphi}}{\varphi_0} \geq 1$. Since $\frac{\widetilde{\varphi}(t)}{\varphi_0(t)} =  1$ on $(0, \epsilon)$ with  $\lim_{t\to \infty}\frac{\widetilde{\varphi}(t)}{\varphi_0(t)} = \lim_{t\to \infty}\frac{\widetilde{\varphi}(t)}{\widetilde{\varphi}(g(t))} = \frac{d}{\widetilde{\varphi}(c)} = \frac{d}{d} = 1$, $\frac{\widetilde{\varphi}(t)}{\varphi_0(t)}$ must attain a maximum value $m > 0$ somewhere on $(0, \infty)$. But then $\widetilde{\varphi}(t) \leq {m}\varphi_0(t)$ for all $t \geq 0$.
\end{proof}

We are now ready to show that for some Banach Function space $L^\rho$, the formal process $a \to a\otimes \varphi_\rho(e^t)$ will in general only succeed in embedding $L^\rho(\M,\tau_\M)$ in $\widetilde{\mathcal{A}}$, if we are dealing with an Orlicz space.

\begin{theorem}\label{214}
Let $\varphi_\rho$ be the fundamental function of the rearrangement invariant Banach Function Space $L^\rho(\mathbb{R})$. Then there exists an Orlicz function $\psi_0$ such that the canonical map $a\otimes \varphi_\rho(e^t)\to a\otimes \varphi_{\psi_0}(e^t)$ yields a homeomorphism (with respect to the topology of convergence in measure) from the subspace $\{a\otimes \varphi_\rho(e^t) \in \widetilde{\mathcal{A}} : a \in \tM\} \subset \widetilde{\mathcal{A}}$ onto $L^{\psi_0}(\M)$.
\end{theorem}

\begin{proof}
Let $\varphi_\rho$ be given. By the preceding theorem we may select a concave function $\varphi_0$ satisfying all the conditions stated in that theorem such that for some $k > 0$ we have that $\frac{1}{k}\varphi_0 \leq \varphi_\rho \leq {k}\varphi_0$. Let $d = \lim_{t\to\infty}\varphi_0(t)$. Now consider the function
$$\psi_0(t) = \left\{\begin{array}{ll}  0 & \qquad 0 \leq t \leq \frac{1}{d} \\ \frac{1}{\varphi_0^{-1}(1/t)} & \qquad \frac{1}{d} < t < \frac{1}{\varphi_0(0+)}\\ \infty & \qquad \frac{1}{\varphi_0(0+)} \leq t \end{array} \right. .$$The conditions on $\varphi_0$ are sufficient to ensure that $\varphi_0^{-1}$ is a well-defined continuous bijection from $(\varphi_0(0+), d)$ to $(0, \infty)$. In addition the 
fact that $\varphi_0$ is both concave and increasing, ensures that $\varphi_0^{-1}$ is in fact convex and increasing on $(\varphi_0(0+), d)$. (To see this apply $\varphi_0^{-1}$ to the inequality $\varphi_0(\lambda s + (1-\lambda)t) \geq \lambda\varphi_0(s) + (1-\lambda)\varphi_0(t)$, and then let $s=\varphi_0^{-1}(u)$ and $t=\varphi_0^{-1}(v)$.) It is an exercise to conclude from this fact that $\frac{1}{\varphi_0^{-1}(1/t)}$ is also increasing on $(\frac{1}{d}, \frac{1}{\varphi_0(0+)})$.  Now notice that the function $g(t)=\frac{1}{t}$ is convex on $(0,\infty)$, and that $g\circ\varphi_0^{-1}\circ g(t) = \frac{1}{\varphi_0^{-1}(1/t)}$ on $(\frac{1}{d}, \frac{1}{\varphi_0(0+)})$. Thus $\frac{1}{\varphi_0^{-1}(1/t)}$ must be convex on $(\frac{1}{d}, \frac{1}{\varphi_0(0+)})$, since on this interval it appears as the composition of three convex functions.

The fact that $\frac{1}{\varphi_0^{-1}(1/t)}$ is convex and increasing on $(\varphi_0(0+), \lim_{t\to\infty}\varphi_0(t))$, in turn suffices to ensure that $\psi_0$ is an Orlicz function. Moreover by construction $$\varphi_0(t) = \frac{1}{\psi_0^{-1}(1/t)}.$$Thus $\varphi_0$ is just the fundamental function of the Orlicz space $L^{\psi_0}(\mathbb{R})$. It follows that $L^{\psi_0}(\M) = \{a\otimes \varphi_0(e^t) \in \widetilde{\mathcal{A}} : a \in \tM\}$. Finally note that the condition $\frac{1}{k}\varphi_0 \leq \varphi_\psi \leq {k}\varphi_0$ ensures that the elements $u = \I \otimes [\varphi_{\psi_0}(e^t)/\varphi_\rho(e^t)]$ and $u^{-1} = \I \otimes [\varphi_\rho(e^t)/\varphi_{\psi_0}(e^t)]$ both belong to $\mathcal{A}$. Hence the map $\widetilde{a} \to u\widetilde{a}$ is a continuous map on $\widetilde{\mathcal{A}}$ with continuous inverse 
$\widetilde{a} \to u^{-1}\widetilde{a}$ which maps $\{a\otimes \varphi_\rho(e^t) \in \widetilde{\mathcal{A}} : a \in \tM\}$ onto $L^{\psi_0}(\M)$ in the prescribed manner.
\end{proof}

\begin{remark}\label{215}
The message from the above theorem is not that there is no room inside $\widetilde{\mathcal{A}}$ for Banach Function spaces, but just that we need to follow a different approach when Orlicz spaces are not in view. The spaces $L^1\cap L^\infty(\M, \tau_\M)$ and $L^1 + L^\infty(\M, \tau_\M)$ are known to be Orlicz spaces with fundamental functions respectively given by $\varphi_{1\cap\infty}(t) = \max(1,t)$ and $\varphi_{1 + \infty}(t) = \min(1,t)$. It is now a simple matter to verify that $w = \I \otimes [\varphi_{1 + \infty}(e^t)/\varphi_{1\cap\infty}(e^t)]$ is a norm one element of $\mathcal{A}$ for which the map $\widetilde{a} \to w\widetilde{a}$ canonically embeds $L^1\cap L^\infty(\M)=\{a\otimes \varphi_{1\cap\infty}(e^t) \in \widetilde{\mathcal{A}} : a \in \tM\}$ into $L^1 + L^\infty(\M)=\{a\otimes \varphi_{1+\infty}(e^t) \in \widetilde{\mathcal{A}} : a \in \tM\}$. Thus we may represent $L^1\cap L^\infty(\M)$ as a subspace of $L^1 + L^\infty(\M)$. The noncommutative rearrangement invariant Banach Function spaces corresponding to $\M$, then appear as exact interpolation spaces of the pair $(L^\infty(\M), L^1(\M))$ \cite[3.2.12]{BS}. 
\end{remark}

\section{Orlicz spaces for type III algebras}

In our attempt to develop a theory of Orlicz spaces for the type III case, we will look at the general and $\sigma$-finite case separately. In this section and the next we therefore assume $\M$ to be a general von Neumann algebra equipped with a faithful normal semifinite weight $\nu$, before returning to $\sigma$-finite algebras in the last section. The reason for distinguishing the two cases like this, is that the $\sigma$-finite case allows for constructs and techniques that seem unworkable in the general case.  A key ingredient in our analysis will be the following consequence of Corollary \ref{Lux} and Theorem \ref{thm:normtop}. Observe that this holds for general von Neumann algebras and seems to have been previously unknown. 

\begin{theorem}\label{thm:uniftop}
Let $\pi(\M)$ be the canonical embedding of $\M$ into the crossed product $\mathcal{A} = \cM{}{}$. We have that $\|\pi(a)\|_\infty = \mu_\epsilon(\pi(a))$ for all $0 < \epsilon \leq 1$. In addition the uniform topology on $\pi(\M)$ is homeomorphic to the topology of convergence in measure that $\pi(\M)$ inherits from $\widetilde{\mathcal{A}}$.
\end{theorem}  

The proof of this result makes use of Haagerup's reduction theorem. We will therefore freely use the notation from \cite{HJX} in our proof. However before applying this theorem, we make three observations.

\textbf{Observation 1:} If $\M$ is semifinite and $\nu$ an fns trace on $\M$, then $\|\pi(a)\|_\infty = \mu_\epsilon(\pi(a))$ for all $0 < \epsilon \leq 1$. To see this notice that we already know from Corollary \ref{Lux} that $\|\pi(a)\|_\infty = \mu_1(\pi(a))$. But if this is the case, we must have $\|\pi(a)\|_\infty = \mu_\epsilon(\pi(a))$ for all $0 < \epsilon \leq 1$.

\textbf{Observation 2:} Let $\nu_1$ and $\nu_2$ be two fns weights on $\M$ and let $\mathcal{A}_1$ and $\mathcal{A}_2$ be the respective crossed products with respect to the modular groups of these two weights. Next let $\kappa: \widetilde{\mathcal{A}}_1 \to \widetilde{\mathcal{A}}_2$ be the map described in \cite[II.37, II.38]{Tp}. Then $$\mu_t(a,\tau_{{\mathcal{A}}_1}) = \mu_t(\kappa(a),\tau_{\mathcal{A}_2}) \quad\mbox{for all}\quad a\in \widetilde{\mathcal{A}}_1, t \geq 0.$$This claim is a fairly straightforward consequence of the fact that $\tau_{{\mathcal{A}}_2} = \tau_{{\mathcal{A}}_1}\circ\kappa^{-1}$ and that $\kappa$ takes projections to projections.

\textbf{Observation 3:} Let $\nu$ be an fns weight on the von Neumann algebra $\M$, and let $E$ be a conditional expectation onto a subalgebra $\M_0$ for which $\nu\circ E = \nu$ and $\nu_0 = \nu|_{\M_0}$ is an fns weight on $\M_0$. Denote the crossed products of these von Neumann algebras with respect to the modular groups of these two weights by $\mathcal{A}$ and $\mathcal{A}_0$ respectively. Then $E$ extends to a conditional expectation $\widetilde{E} : \mathcal{A} \to \mathcal{A}_0$ for which we have that $\tau_{\mathcal{A}} = 
\tau_{\mathcal{A}_0}\circ\widetilde{E}$. For any $p \geq 1$ we then have that $L^p(\M_0) \subset L^p(\M)$. So given $a\in L^p(\M_0)$, we must have that $\mu_t(a,\widetilde{\mathcal{A}}_0) = \mu_t(a,\widetilde{\mathcal{A}})$ for all $t \geq 0$. This claim follows from the above trace formula, and the fact that each $\mu_t$ is determined by the spectral projections of $a$, all of which must lie in $\mathcal{A}_0\subset \mathcal{A}$, since $a$ belongs to $L^p(\M_0)\subset\widetilde{\mathcal{A}_0}$.

\begin{proof}[Proof of Theorem \ref{thm:uniftop}]We now use the above observations and the reduction theorem, to prove the formula for the norm. Here $\mathcal{R}$ denotes the crossed product of $\M$ with the diadic rationals. In view of the fact that we are taking a 
crossed product with a discrete group, there exists a faithful normal conditional expectation $\Phi$ from $\mathcal{R}$ onto $\M$. Additionally there also exists an increasing family $\{\mathcal{R}_i\}_{i\in I}$ of w*-closed involutive subalgebras of $\mathcal{R}$ 
with union w*-dense in $\M$. Each of these admits a faithful normal tracial state. For each $i$ there exists a normal conditional expectation $\Phi_i$ from $\mathcal{R}$ onto $\mathcal{R}_i$ satisfying $\Phi_i\circ\Phi_j=\Phi_j\circ\Phi_i=\Phi_i$ whenever $i\leq j$. 
By $\widehat{\nu}$ we will mean the dual weight on $\mathcal{R}$ given by $\nu\circ\Phi$. 

By Observations 1 and 2, the result holds for each $(\mathcal{R}_i; \widehat{\nu}|_{\mathcal{R}_i})$. But by Observation 3, it then also holds on the subspace $(\cup\mathcal{R}_i; \widehat{\nu})$ of $(\mathcal{R}; \widehat{\nu})$. We may now use order and \cite[Proposition 1.7]{DDdP3} to conclude that the formula holds for $(\mathcal{R}_+; \widehat{\nu})$. Since for any $t$ and any $a \in \mathcal{R}$ we have that $\mu_t(a) = \mu_t(|a|)$, the formula then also holds for $(\mathcal{R}; \widehat{\nu})$. But by Observation 3, it must then hold for $(\M, \nu)$.

It remains to prove the claim regarding the equivalence of the two topologies. It is known that for sequences in $\pi(\M)$ convergence in norm implies convergence in measure. The converse is a direct consequence of the norm formula applied to the fact that convergence in the measure topology may be described in terms of the singular function.
\end{proof}

\begin{lemma}
Let $\psi$ be an Orlicz function. Then there exists a constant $k_\psi > 0$ so that $\varphi_\psi(t) \leq k_\psi(1+t)$ for all $t\geq 0$ (respectively $\widetilde{\varphi}_\psi(t) \leq k_\psi(1+t)$ for all $t\geq 0$).
\end{lemma}

\begin{proof}
Since $\varphi_\psi$ is quasi-concave, it is continuous and increasing on $(0,\infty)$, whereas $t\to \frac{\varphi_\psi(t)}{t}$ is continuous and decreasing on $(0,\infty)$. Using these facts it is now an exercise to show that $\varphi_\psi(t) \leq \varphi_\psi(1)\max(1,t) \leq \varphi_\psi(1)(1+t)$.  
\end{proof}

\begin{proposition}\label{prop:meas}
For any $a\in\pi(\mathfrak{n}_\nu)$ the operators $a\varphi_\psi(h)^{1/2},\, \varphi_\psi(h)^{1/2}a$ (respectively $a\widetilde{\varphi}_\psi(h)^{1/2},\, \widetilde{\varphi}_\psi(h)^{1/2}a$) are $\tau_{\mathcal{A}}$-pre-measurable.
\end{proposition}

\begin{proof}
For any $a\in\pi(\mathfrak{n}_\nu)$ it follows from the lemma that $$|\varphi_\psi(h)^{1/2}a|^2 = a^*\varphi_\psi(h)a \leq k_\psi(a^*a +a^*ha).$$The claim may now be deduced from the fact that $a^*a +a^*ha$ is known to be 
$\tau_{\mathcal{A}}$-measurable (see \cite{GL2}).
\end{proof}

The formal prescription for defining Orlicz spaces for type III algebras was already hinted at in Remark \ref{rem:b}. In principle one simply replaces $\I\otimes e^t$ with the density $h = \frac{d\widetilde{\nu}}{d\tau_{\mathcal{A}}}$ in the equivalent conditions stated there. The above Proposition, now enables us to give rigorous meaning to these ideas.
 
\begin{definition}\label{deforldef}
Let $\psi$ be an Orlicz function. We define the Orlicz space corresponding to the Luxemburg norm to be $$L^\psi(\M) = \{a\in\widetilde{\mathcal{A}} : [e\widetilde{\varphi}_{\psi^*}(h)^{1/2}]a[\widetilde{\varphi}_{\psi^*}(h)^{1/2}f] \in L^1(\M)\mbox{ for all projections }e,f\in \mathfrak{n}_\nu\}.$$ 

In a similar fashion the Orlicz space corresponding to the Orlicz norm is defined to be 
$$L_\psi(\M) =\{a\in\widetilde{\mathcal{A}} : [e{\varphi}_{\psi^*}(h)^{1/2}]a[{\varphi}_{\psi^*}(h)^{1/2}f] \in L^1(\M)\mbox{ for all projections }e,f\in \mathfrak{n}_\nu\}.$$
\end{definition} 

\begin{remark}
Observe that the Orlicz spaces as defined above are large spaces. Consider the left ideals in $\M$ defined by$$\mathfrak{n}_\psi =\{a\in\M : a{\varphi}_{\psi}(h)^{1/2}\mbox{ is closable and } [a{\varphi}_{\psi}(h)^{1/2}] \in \widetilde{\mathcal{A}}\}$$ $$\mathfrak{n}_{\widetilde{\psi}} =\{a\in\M : a\widetilde{\varphi}_{\psi}(h)^{1/2}\mbox{ is closable and } [a\widetilde{\varphi}_{\psi}(h)^{1/2}]\in \widetilde{\mathcal{A}}\}.$$Now consider the $*$-subalgebras $\mathfrak{m}_\psi = \mathrm{span}[(\mathfrak{n}_\psi)^*\mathfrak{n}_\psi]$ and $\mathfrak{m}_{\widetilde{\psi}} = \mathrm{span}[(\mathfrak{n}_{\widetilde{\psi}})^*\mathfrak{n}_{\widetilde{\psi}}]$. It is a simple matter to verify that ${\varphi}_{\psi}(h)^{1/2}\mathfrak{m}_\psi{\varphi}_{\psi}(h)^{1/2} \subset L^\psi(\M)$ and $\widetilde{\varphi}_{\psi}(h)^{1/2}\mathfrak{m}_{\widetilde{\psi}}\widetilde{\varphi}_{\psi}(h)^{1/2} \subset L_\psi(\M)$. Since by Proposition \ref{prop:meas} we have that $\mathfrak{m}_\nu \subset [\mathfrak{m}_\psi \cap \mathfrak{m}_{\widetilde{\psi}}]$, it is clear that the spaces as defined above are therefore not trivial.
\end{remark}

To proceed with the development of the theory, we need three technical lemmata. 

\begin{lemma}\label{tech1}
Let $\{a_\lambda\}$ be a net in $\mathcal{A}$ which converges strongly to some $a_0\in \mathcal{A}$. For any $b, c \in \widetilde{\mathcal{A}}$, the net $\{ca_\lambda b\}$ will then converge strongly to $ca_0b$. 
\end{lemma}

\begin{proof}
Let $c=v_c|c|$, $b=v_b|b|$ be the polar decompositions of $c$ and $b$ respectively. Then of course $c^*=v^*_c|c^*|$ and $c=|c^*|v_c$. The net $\{v_ca_\lambda v_b\}$ clearly converges strongly to $v_ca_0v_b$. So to prove the lemma we may replace $\{a_\lambda\}$ with $\{v_ca_\lambda v_b\}$, and $b$ and $c$ with $|b|$ and $|c^*|$. In other words we may assume $b$ and $c$ to be positive. Now let $\chi_{[0,n]}$ and $\widetilde{\chi}_{[0,m]}$ be the spectral projections of $b$ and $c$ respectively, corresponding to the intervals $[0,n]$ and $[0,m]$. Fixing $m$ we have that $\widetilde{\chi}_{[0,m]}ca_\lambda b\chi_{[0,n]}$ converges strongly to $\widetilde{\chi}_{[0,m]}ca_0b\chi_{[0,n]}$ for any $n$. Let $\mathfrak{h}$ denote the canonical Hilbert space on which $\mathcal{A}$ is realised as an algebra of operators. Since $\cup_{n\in\mathbb{N}}\chi_{[0,n]}(\mathfrak{h})$ is $\tau_{\mathcal{A}}$-dense in $\mathfrak{h}$, we have that $\{\widetilde{\chi}_{[0,m]}ca_\lambda b\}$ converges strongly to $\widetilde{\chi}_{[0,m]}ca_0b$ for any $m$ \cite[I.12]{Tp}. 

If now we pass to the adjoints we see from the above that $(\lim a_\lambda b)^* c$ agrees with $ba_0^*c$ on the 
$\tau_{\mathcal{A}}$-dense subspace $\cup_{m\in\mathbb{N}}\widetilde{\chi}_{[0,m]}(\mathfrak{h})$ of $\mathfrak{h}$, and hence that $(\lim a_\lambda b)^* c = ba_0^*c$. In other words  $c(\lim a_\lambda b) = ca_0b$.
\end{proof}

Our first application of this lemma yields an alternative way of defining these spaces.

\begin{proposition}\label{orldef}
Let $\psi$ be an Orlicz function. Given $a\in \widetilde{\mathcal{A}}$, we have that $a\in L^\psi(\M)$ if and only if $[b_0\widetilde{\varphi}_{\psi^*}(h)^{1/2}]a[\widetilde{\varphi}_{\psi^*}(h)^{1/2}b_1^*] \in L^1(\M)$ for any $b_0,b_1\in \mathfrak{n}_\nu$. (Similarly $a\in L_\psi(\M)$ if and only if $[b_0{\varphi}_{\psi^*}(h)^{1/2}]a[{\varphi}_{\psi^*}(h)^{1/2}b_1^*] \in L^1(\M)$ for any 
$b_0,b_1\in \mathfrak{n}_\nu$.) 
\end{proposition}

\begin{proof}
The ``if'' part is trivial. To see the ``only if'' part, suppose we are given an arbitrary positive element $b\in (\mathfrak{n}_\nu)^+$, and some $a \in L^\psi(\M)$. Given any $0< \epsilon$, the fact the spectral projection $\chi_{I_\epsilon}(b)$ of $b$ corresponding to the interval $I_\epsilon=[\epsilon,\infty)$ is dominated by $(\epsilon)^{-1}b$, ensures that $\nu(\chi_{I_\epsilon}(b)) \leq (\epsilon)^{-2}\nu(b^2) < \infty$, and hence that $\chi_{I_\epsilon}(b) \in \mathfrak{n}_\nu$. So by hypothesis $$[\chi_{I_\epsilon}\widetilde{\varphi}_{\psi^*}(h)^{1/2}]a[\widetilde{\varphi}_{\psi^*}(h)^{1/2}e]\in L^1(\M)$$for any $\epsilon > 0$ and any projection $e \in \mathfrak{n}_\nu$. But then given $s\in \mathbb{R}$, we have that  
\begin{eqnarray*}
\chi_{I_\epsilon}(b)\theta_s([b\widetilde{\varphi}_{\psi^*}(h)^{1/2}]a[\widetilde{\varphi}_{\psi^*}(h)^{1/2}e]) &=&
\theta_s(\chi_{I_\epsilon}(b)[b\widetilde{\varphi}_{\psi^*}(h)^{1/2}]a[\widetilde{\varphi}_{\psi^*}(h)^{1/2}e])\\
&=& \theta_s([(b\chi_{I_\epsilon}(b))\widetilde{\varphi}_{\psi^*}(h)^{1/2}]a[\widetilde{\varphi}_{\psi^*}(h)^{1/2}e])\\
&=& b\theta_s([\chi_{I_\epsilon}(b)\widetilde{\varphi}_{\psi^*}(h)^{1/2}]a[\widetilde{\varphi}_{\psi^*}(h)^{1/2}e])\\
&=& e^{-s}b[\chi_{I_\epsilon}(b)\widetilde{\varphi}_{\psi^*}(h)^{1/2}]a[\widetilde{\varphi}_{\psi^*}(h)^{1/2}e]\\
&=& e^{-s}[b\chi_{I_\epsilon}(b)\widetilde{\varphi}_{\psi^*}(h)^{1/2}]a[\widetilde{\varphi}_{\psi^*}(h)^{1/2}e]\\
&=& e^{-s}\chi_{I_\epsilon}(b)[b\widetilde{\varphi}_{\psi^*}(h)^{1/2}]a[\widetilde{\varphi}_{\psi^*}(h)^{1/2}e].
\end{eqnarray*}
If now we let $\epsilon$ decrease to zero it will follow that $\theta_s([b\widetilde{\varphi}_{\psi^*}(h)^{1/2}]a[\widetilde{\varphi}_{\psi^*}(h)^{1/2}e]) = e^{-s}[b\widetilde{\varphi}_{\psi^*}(h)^{1/2}]a[\widetilde{\varphi}_{\psi^*}(h)^{1/2}e]$ for all $s\in \mathbb{R}$. This is in turn enough to ensure that $$[b\widetilde{\varphi}_{\psi^*}(h)^{1/2}]a[\widetilde{\varphi}_{\psi^*}(h)^{1/2}e]\in L^1(\M)\quad\mbox{for any}\quad b\in \mathfrak{n}_\nu, e\in \mathbb{P}(\mathfrak{n}_\nu).$$A similar argument as before now suffices to show that this condition forces the fact that $[b_0\widetilde{\varphi}_{\psi^*}(h)^{1/2}]a[\widetilde{\varphi}_{\psi^*}(h)^{1/2}b_1^*] \in L^1(\M)$ for any $b_0,b_1\in \mathfrak{n}_\nu$.
\end{proof}

By Exercise 2.4.8(d) of \cite{S}, the semifiniteness of the weight $\nu$ guarantees the existence of a net $\{f_\alpha\}\subset(\mathfrak{n}_\nu)^+$ which increases monotonically to $\I$. However for our purposes it will be more convenient to organise matters so that $\{f_\alpha^2\}$ increases to $\I$. The existence of such a net is the content of our second technical lemma.

\begin{lemma}\label{tech2}
There exists a net $\{f_\alpha\}\subset\mathfrak{n}_\nu^+$ such that $\{f_\alpha^2\}$ increases monotonically to $\I$.
\end{lemma}

\begin{proof}
By the normality of the weight $\nu$, we may select a net $\{\omega_\alpha\} \subset (\M_*)^+$  which increases to $\nu$. Since $\nu$ is also semifinite, we use \cite[Proposition II.4]{Tp} to conclude that the densities $h_\alpha=\frac{d\widetilde{\omega_\alpha}}{d\tau_{\mathcal{A}}}$ increase monotonically to $h = \frac{d\widetilde{\nu}}{d\tau_{\mathcal{A}}}$. By \cite[Proposition 3.2]{GL2} each $h_\alpha$ is of the form $h_\alpha = h^{1/2}f_\alpha^*f_\alpha h^{1/2}$ for some $f_\alpha$ in the unit ball of $\mathfrak{n}_\nu$. On replacing $f_\alpha$ with $|f_\alpha|$ if necessary, we may assume without loss of generality that $f_\alpha \geq 0$. It remains to prove that $\{f_\alpha^2\}$ increases monotonically to $\I$.

For any $a\in (\mathfrak{m}_\nu)^+$ and any $\alpha \geq \beta$, we have that 
$$\tr((f_\alpha^2-f_\beta^2)[h^{1/2}ah^{1/2}])= \tr([h^{1/2}(f_\alpha^2-f_\beta^2)h^{1/2}]a)= \tr((h_\alpha-h_\beta)a)\geq 0.$$Since 
$\{[h^{1/2}ah^{1/2}] : a\in (\mathfrak{m}_\nu)^+\}$ is dense in $L^1(\M)^+$, it follows that $f_\alpha^2-f_\beta^2\geq 0$, i.e. that $\{f_\lambda^2\}$ 
is increasing. Now let $z=\sup f_\lambda^2 \leq \I$. Then for any $a\in (\mathfrak{m}_\nu)^+$ we have 
\begin{eqnarray*}
\tr(z[h^{1/2}ah^{1/2}])&=& \sup_\alpha\tr(f_\alpha^2[h^{1/2}ah^{1/2}])\\
&=&\sup_\alpha\tr(a^{1/2}[h^{1/2}f_\alpha^2h^{1/2}]a^{1/2})\\
&=&\sup_\alpha\tr(a^{1/2}h_\alpha a^{1/2})\\
&=&\tr(a^{1/2}ha^{1/2})\\
&=&\tr([h^{1/2}a^{1/2}h^{1/2}]).
\end{eqnarray*}
Again using the fact that $\{[h^{1/2}ah^{1/2}] : a\in (\mathfrak{m}_\nu)^+\}$ is dense in $L^1(\M)^+$, we may conclude from the above equality that $z=\I$.

(To see that $\{[h^{1/2}ah^{1/2}] : a\in (\mathfrak{m}_\nu)^+\}$ is dense in $L^1(\M)^+$, recall that any $b \in L^1(\M)^+$ may be written in the form 
$b=c^*c$ for some $c\in L^2(\M)$ and that any $c\in L^2(\M)$ is the norm-limit of terms of the form $[dh^{1/2}]$ where $d\in \mathfrak{n}_\nu$.)
\end{proof}
 
Equipped with this lemma, we have the technology to establish our final technical lemma.

\begin{lemma}\label{arblemma}
Let $a\in \widetilde{\mathcal{A}}$ be given. If $[e\widetilde{\varphi}_{\psi^*}(h)^{1/2}]a[\widetilde{\varphi}_{\psi^*}(h)^{1/2}f] = 0$ for all projections $e,f\in \mathfrak{n}_\nu$, then $a=0$.
\end{lemma}

\begin{proof}
Suppose that indeed we do have that $[e\widetilde{\varphi}_{\psi^*}(h)^{1/2}]a[\widetilde{\varphi}_{\psi^*}(h)^{1/2}f] = 0$ for all projections $e,f\in \mathfrak{n}_\nu$. A similar argument to the one employed in Proposition \ref{orldef} now shows that then we in fact have that $[b\widetilde{\varphi}_{\psi^*}(h)^{1/2}]a[\widetilde{\varphi}_{\psi^*}(h)^{1/2}c^*] = 0$ for all $b,c\in \mathfrak{n}_\nu$. Now let $\{f_\alpha\}\subset\mathfrak{n}_\nu^+$ be a net for which $\{f_\alpha^2\}$ increases monotonically to $\I$. We show that the left supports $s_L(\widetilde{\varphi}_{\psi^*}(h)^{1/2}f_\alpha)$ must then increase monotonically to $s_L(\widetilde{\varphi}_{\psi^*}(h)^{1/2})=\I$; that is $\I = \vee_\alpha s_L(\widetilde{\varphi}_{\psi^*}(h)^{1/2}f_\alpha)$. To see this notice that $s_L(\widetilde{\varphi}_{\psi^*}(h)^{1/2}f_\alpha)=\mathrm{supp}([\widetilde{\varphi}_{\psi^*}(h)^{1/2}f_\alpha^2\widetilde{\varphi}_{\psi^*}(h)^{1/2}])$. Now let $\rho_\alpha$ and $\rho$ be the normal weights associated with $[\widetilde{\varphi}_{\psi^*}(h)^{1/2}f_\alpha^2\widetilde{\varphi}_{\psi^*}(h)^{1/2}]$ and $\widetilde{\varphi}_{\psi^*}(h)$ respectively, as described in \cite[Theorem 1.12]{Ha}. Since by construction $[\widetilde{\varphi}_{\psi^*}(h)^{1/2}f_\alpha^2\widetilde{\varphi}_{\psi^*}(h)^{1/2}]$ increases monotonically to $\widetilde{\varphi}_{\psi^*}(h)$, $\rho_\alpha$ must increase to $\rho$ by \cite[Theorem 1.12]{Ha}. But then $\mathrm{supp}(\rho_\alpha)$ increases to $\mathrm{supp}(\rho)$. Since $\mathrm{supp}(\rho_\alpha)=\mathrm{supp}([\widetilde{\varphi}_{\psi^*}(h)^{1/2}f_\alpha^2\widetilde{\varphi}_{\psi^*}(h)^{1/2}])$ and $\mathrm{supp}(\rho_\alpha)=\mathrm{supp}(\widetilde{\varphi}_{\psi^*}(h))$, the claim follows.

If we now combine the fact that $\I = \vee_\alpha s_L(\widetilde{\varphi}_{\psi^*}(h)^{1/2}f_\alpha)$ with the observation that $[b\widetilde{\varphi}_{\psi^*}(h)^{1/2}]a[\widetilde{\varphi}_{\psi^*}(h)^{1/2}f_\alpha] = 0$ for all $\alpha$ and all $b\in \mathfrak{n}_\nu$, it follows that $$[b\widetilde{\varphi}_{\psi^*}(h)^{1/2}]a = 0\quad\mbox{for all}\quad b\in \mathfrak{n}_\nu.$$A similar argument as before using right supports instead of left supports, now shows that this latter condition in turn ensures that $a=0$ as required.
\end{proof}

Having defined Orlicz spaces for type III algebras, it is incumbent on us to find an appropriate topology for these spaces. The theory for semifinite algebras suggests the quantity $a\to\mu_1(a)$ as a natural candidate for a norm. We will shortly see that this quantity yields a quasi-norm on each of these spaces, with the quasi-norm topology equivalent to the induced topology of convergence in measure. Although we believe the topology of convergence in measure to actually be normable on each of these Orlicz spaces, we are at present only able to confirm this suspicion for the class of Orlicz spaces having an upper Boyd index strictly less than 1. (For an exposition on the Boyd indices of rearrangement-invariant Banach Function spaces, refer to the book of Bennett and Sharpley \cite{BS}.) In the next section and the one after that we will initiate an investigation of the normability of general Orlicz spaces. Our strategy will be to first try and prove that the ``largest'' and ``smallest'' of all noncommutative Orlicz spaces are normable, in the hope that at some future stage this conclusion may be extended to all the intermediate noncommutative Orlicz spaces by some interpolative procedure. However as we shall see, for now we are only partially successful in this endeavour.

\begin{theorem}\label{mu1}
Let $\psi$ be an Orlicz function. For any $a$ in either $L^\psi(\M)$ or $L_\psi(\M)$, we then have that $$t\mu_t(a) \leq \mu_1(a) \qquad\mbox{for all} \qquad 0<t\leq 1.$$ 
\end{theorem}
 
\begin{proof}
The proofs for the two cases are entirely analogous. Hence it suffices to prove the theorem for space $L^\psi(\M)$. So let us assume that $a \in L^\psi(\M)$, and let $0<t\leq1$ be given. We may of course write such a $t$ as $t=e^s$ where $s\leq 0$. Observe that for any $r>0$, we then have that
\begin{eqnarray*}
\frac{1}{t}\lambda_r(a) &=& e^{-s}\tau_{\mathcal{A}}(\chi_{(r,\infty )}(|a|))\\
&=& \tau_{\mathcal{A}}(\theta_s (\chi_{(r,\infty )}(|a|)))\\
&=& \tau_{\mathcal{A}}(\chi_{(r,\infty )}(\theta_s|a|))\\
&=& \tau_{\mathcal{A}}(\chi_{(r,\infty )}(|\theta_s(a)|))\\
&=& \lambda_r(\theta_s(a)).
\end{eqnarray*}
On taking infimums, this equality now yields the conclusion that $$\mu_t(a) = \mu_1(\theta_s(a)) \qquad\mbox{for all}\qquad t=e^s, \, s\leq 0.$$

For any $b, c \in \widetilde{\mathcal{A}}$ with $|b|\leq \I$, it is a simple matter to see that $$\mu_1(bc) = \mu_1(c^*|b|^2c)^{1/2} \leq \mu_1(c^*c)^{1/2} = \mu(c)$$and similarly that $\mu_1(cb) \leq \mu_1(c)$. With $s$ and $t$ as before, consider the operator $d_t = \widetilde{\varphi}_{\psi^*}(\tfrac{1}{t}h)^{-1}\widetilde{\varphi}_{\psi^*}(h)$. Since $\widetilde{\varphi}_{\psi^*}$ is increasing and $0< t \leq 1$, we clearly have that $\widetilde{\varphi}_{\psi^*}(\tfrac{1}{t}h) \geq \widetilde{\varphi}_{\psi^*}(h)$ by the Borel functional calculus. In other words $|d_t| \leq \I$. The theorem will now follow from this fact once we have established the following identity: $$\theta_s(a) = \frac{1}{t}[d_t^{1/2}ad_t^{1/2}].$$Assuming this identity to hold for now, it is fairly easy to see that the above (in)equalities imply that 
\begin{equation}\label{eq:muineq}
t\mu_t(a) = t\mu_1(\theta_s(a)) = \mu_1(d_t^{1/2}ad_t^{1/2}) \leq \mu_1(a).
\end{equation}
It therefore remains to prove that $\theta_s(a) = \frac{1}{t}[d_t^{1/2}ad_t^{1/2}]$.

For $t=e^s$ and $b_0,b_1\in\mathfrak{n}_\nu$ we have that
\begin{eqnarray*}
\frac{1}{t}[b_0\widetilde{\varphi}_{\psi^*}(\frac{1}{t}h)^{1/2}](d_t^{1/2}ad_t^{1/2})[\widetilde{\varphi}_{\psi^*}(\frac{1}{t}h)^{1/2}b_1^*] &=&
 \frac{1}{t}([b_0\widetilde{\varphi}_{\psi^*}(\frac{1}{t}h)^{1/2}]d_t^{1/2})a(d_t^{1/2}[\widetilde{\varphi}_{\psi^*}(\frac{1}{t}h)^{1/2}b_1^*])\\
&=& \frac{1}{t}[b_0\widetilde{\varphi}_{\psi^*}(h)^{1/2}]a[\widetilde{\varphi}_{\psi^*}(h)^{1/2}b_1^*]\\
&=& e^{-s}[b_0\widetilde{\varphi}_{\psi^*}(h)^{1/2}]a[\widetilde{\varphi}_{\psi^*}(h)^{1/2}b_1^*]\\
&=& \theta_s([b_0\widetilde{\varphi}_{\psi^*}(h)^{1/2}]a[\widetilde{\varphi}_{\psi^*}(h)^{1/2}b_1^*])\\
&=& \theta_s([b_0\widetilde{\varphi}_{\psi^*}(h)^{1/2}])\theta_s(a)\theta_s([\widetilde{\varphi}_{\psi^*}(h)^{1/2}b_1^*])\\
&=& [b_0\theta_s(\widetilde{\varphi}_{\psi^*}(h)^{1/2})]\theta_s(a)[\theta_s(\widetilde{\varphi}_{\psi^*}(h)^{1/2})b_1^*]\\
&=& [b_0\widetilde{\varphi}_{\psi^*}(e_{-s}h)^{1/2})]\theta_s(a)[\widetilde{\varphi}_{\psi^*}(e^{-s}h)^{1/2}b_1^*]\\
&=& [b_0\widetilde{\varphi}_{\psi^*}(\frac{1}{t}h)^{1/2})]\theta_s(a)[\widetilde{\varphi}_{\psi^*}(\frac{1}{t}h)^{1/2}b_1^*].
\end{eqnarray*}
The preceding lemma now shows that $\frac{1}{t}[d_t^{1/2}ad_t^{1/2}] - \theta_s(a)=0$, as required. 
\end{proof} 
 
As a consequence of the above theorem, we have the following important proposition.

\begin{proposition}\label{qntop}
The quantity $\mu_1(\cdot)$ is a quasinorm for both $L^\psi(\M)$ and $L_\psi(\M)$. The topology induced on these spaces by this quasinorm is complete and is homeomorphic to the topology of convergence in measure inherited from $\widetilde{\mathcal{A}}$. 
\end{proposition} 

\begin{proof}
Once again the analogy of the arguments used in the two cases, makes it sufficient to consider only the spaces $L^\psi(\M)$.

First suppose that we are given $a \in L^\psi(\M)$ with $\mu_1(a)=0$. Given any two projections $e,f\in \pi(\mathfrak{n}_\nu)$, we have that $b = e\widetilde{\varphi}_{\psi^*}(h)^{1/2}a\widetilde{\varphi}_{\psi^*}(h)^{1/2}f \in L^1(\M)$. Consequently  
\begin{eqnarray*}
\|[e\widetilde{\varphi}_{\psi^*}(h)^{1/2}]a[\widetilde{\varphi}_{\psi^*}(h)^{1/2}f]\|_1 
&=& 3\mu_3([e\widetilde{\varphi}_{\psi^*}(h)^{1/2}]a[\widetilde{\varphi}_{\psi^*}(h)^{1/2}f]) \\ 
& \leq & 3\mu_1(e\widetilde{\varphi}_{\psi^*}(h)^{1/2})\mu_1(a)\mu_1(\widetilde{\varphi}_{\psi^*}(h)^{1/2}f) \\ 
&=& 0.
\end{eqnarray*}
So we clearly have that $[e\widetilde{\varphi}_{\psi^*}(h)^{1/2}]a[\widetilde{\varphi}_{\psi^*}(h)^{1/2}f]=0$ for any two projections $e,f\in \pi(\mathfrak{n}_\nu)$. Hence $a=0$ by Lemma\ref{arblemma}.

Trivially $\mu_1(\lambda a) = |\lambda|\mu_1(a)$ for any scalar $\lambda$ and any $a \in L^\psi(\M)$. It remains to show that $\mu_1$ satisfies a generalised triangle inequality. Given $a,b \in L^\psi(\M)$, we may conclude from the properties of $\mu_t$  and the preceding theorem that $$\mu_1(a+b) \leq 2\left(\frac{1}{2}\mu_{1/2}(a)+\frac{1}{2}\mu_{1/2}(b)\right) \leq 2(\mu_1(a)+\mu_1(b)).$$

We now show that $L^\psi(\M)$ is a closed subspace of $\widetilde{\mathcal{A}}$. Once we have established that the topology on $L^\psi(\M)$ induced by $\mu_1$ agrees with the topology of convergence in measure, this closedness will suffice to prove the completeness of $L^\psi(\M)$. Since for any $b\in\widetilde{\mathcal{A}}$ we have that $b\in L^1(\M)$ if and only if $\theta_s(b) = e^{-s}b$, membership of an element $a$ of $\widetilde{\mathcal{A}}$ to $L^\psi(\M)$ can be rephrased as the claim that $$\theta_s(e\widetilde{\varphi}_{\psi^*}(h)^{1/2}a\widetilde{\varphi}_{\psi^*}(h)^{1/2}f) = e^{-s}e\widetilde{\varphi}_{\psi^*}(h)^{1/2}a\widetilde{\varphi}_{\psi^*}(h)^{1/2}f$$for all $s \in \mathbb{R}$ and all $e,f \in \mathbb{P}(\mathfrak{n}_\nu)$. Equivalently this boils down to saying that $a$ is in the intersection of the kernels of the operators $$a \to \theta_s(e\widetilde{\varphi}_{\psi^*}(h)^{1/2}a\widetilde{\varphi}_{\psi^*}(h)^{1/2}f) - e^{-s}e\widetilde{\varphi}_{\psi^*}(h)^{1/2}a\widetilde{\varphi}_{\psi^*}(h)^{1/2}f \qquad s\in \mathbb{R},\quad e,f \in \mathbb{P}(\mathfrak{n}_\nu)$$all of which are continuous in the topology of convergence in measure. The kernels therefore being closed, this suffices to prove the claim regarding the closedness of $L^\psi(\M)$.

It remains to prove that the topology induced on $L^\psi(\M)$ by $\mu_1$ is precisely the topology of convergence in measure. The fact that convergence in measure implies convergence in the quasi-norm $\mu_1$, follows from the description of convergence in measure in terms of $\mu_t$. For the converse fix $0 < \epsilon \leq 1$, and suppose we are given some $a\in L^\psi(\M)$ with $\mu_1(a)<\epsilon^2$. By the preceding theorem, we have that $\mu_\epsilon(a) \leq \frac{1}{\epsilon}\mu_1(a) <\epsilon$. Since the map $s\to \lambda_s(a)$ is non-increasing, it follows that $\lambda_\epsilon(a) \leq \lambda_{\mu_\epsilon}(a)\leq \epsilon$. But by \cite[I.7]{Tp}, this suffices to ensure that $a$ belongs to the basic neighbourhood of zero $\mathcal{N}(\epsilon,\epsilon)$. Thus any sequence in $L^\psi(\M)$ that converges to zero in the quasinorm $\mu_1$, converges to zero in measure.
\end{proof}

\begin{remark}\label{lp}
The $L^p$-spaces with $p\geq 1$ are known to be Orlicz spaces induced by the Orlicz functions $\psi_p(t)=t^p$ in the case $\infty>p\geq 1$, and by the function $\psi_\infty$ defined to be $0$ on $[0,1]$, and infinite valued on $(1,\infty)$ in the case $p=\infty$. These in turn yield fundamental functions of the form $\varphi_p(t)=t^{1/p}$ on $(0,\infty)$ for ($\infty\geq p\geq 1$). Now let $q$ be the conjugate index of $p$. For any $a\in\widetilde{\mathcal{A}}$ and any two projections $e,f \in \mathfrak{n}_\nu$ we  have that 
\begin{eqnarray*}\theta_s ([e{\varphi}_{q}(h)^{1/2}]a[{\varphi}_{q}(h)^{1/2}f])&=& 
[e(\theta_s (h))^{1/(2q)}]\theta_s(a)[(\theta_s (h))^{1/(2q)}f]\\
&=& [e(e^{-s}h))^{1/(2q)}]\theta_s(a)[(e^{-s}h)^{1/(2q)}f]\\
&=& e{-s/q}[eh^{1/(2q)}]\theta_s(a)[h^{1/(2q)}f]\\
&=& e{-s/q}[e{\varphi}_{q}(h)^{1/2}]\theta_s(a)[{\varphi}_{q}(h)^{1/2}f]
\end{eqnarray*} 
for all $s\in\mathbb{R}$. The K\"othe dual of $L^p(\mathbb{R})$ is of course just $L^q(\mathbb{R})$ where $q$ is the conjugate index of $p$. Thus the prescription for membership of $L^p(\M)$ given by Definition \ref{deforldef} is that $[e{\varphi}_{q}(h)^{1/2}]a[{\varphi}_{q}(h)^{1/2}f]$ should belong to $L^1(\M)$ for any two projections $e,f \in \mathfrak{n}_\nu$. This in turn can be rephrased as the statement that $$\theta_s([e{\varphi}_{q}(h)^{1/2}]a[{\varphi}_{q}(h)^{1/2}f])=e^{-s}[e{\varphi}_{q}(h)^{1/2}]a[{\varphi}_{q}(h)^{1/2}f]$$for all $s\in \mathbb{R}$. On combining the previous two centred formulae, it is clear that $a\in\widetilde{\mathcal{A}}$ fulfils the prescription for membership of $L^p(\M)$ as given in Definition \ref{deforldef}, if and only if 
\begin{eqnarray*}[e{\varphi}_{q}(h)^{1/2}]\theta_s(a)[{\varphi}_{q}(h)^{1/2}f]&=& e^{-s+s/q}[e{\varphi}_{q}(h)^{1/2}]a[{\varphi}_{q}(h)^{1/2}f]\\
&=& e^{-s/p}[e{\varphi}_{q}(h)^{1/2}]a[{\varphi}_{q}(h)^{1/2}f]
\end{eqnarray*}
for all $s\in\mathbb{R}$ and all projections $e,f \in \mathfrak{n}_\nu$. By Lemma \ref{arblemma}, this is equivalent to the claim that $\theta_s(a)=e^{-s/p}a$ for all $s\in \mathbb{R}$, which is the standard prescription for identifying the elements of $L^p(\M)$. Thus Definition \ref{deforldef} in a very natural way dovetails with the well-known way of defining Haagerup $L^p$-spaces. In the case $\infty>p\geq 1$ it is also a well-known fact that the elements of these $L^p$-spaces satisfy the property that $\|a\|_p=t^{1/p}\mu_t(a)$ for any $t>0$. On combining this with Theorem \ref{thm:uniftop}, it is then clear that for these spaces the quantity $\mu_1(a)=\sup_{0<t\leq 1}t\mu_t(a)$ is just the usual $L^p$-norm. The theory on Orlicz spaces presented here, may therefore be seen as a natural extension of the theory of Haagerup $L^p$-spaces.
\end{remark}

We next show that for a very large class of Orlicz functions, the topology on the constructed spaces is actually normable. (As noted earlier, the reader should refer to the book of Bennett and Sharpley \cite{BS} for an exposition on the Boyd indices.) 

\begin{theorem}\label{ntop}
Let $\psi$ be an Orlicz function. The topology of convergence in measure on $L^\psi(\M)$ (respectively $L_\psi(\M)$) is normable whenever the upper Boyd index of $L^\psi(\mathbb{R})$ (respectively $L_\psi(\mathbb{R})$) is strictly less than 1.
\end{theorem}
 
\begin{proof}
We have already noted that $\varphi_\psi(t) \leq \widetilde{\varphi}_\psi(t) \leq 2\varphi_\psi(t)$ for all $t \geq 0$. Hence both $\beta(h)= \varphi_\psi(h)(\widetilde{\varphi}_\psi(h))^{-1}$ and its inverse $\widetilde{\varphi}_\psi(h)(\varphi_\psi(h))^{-1}$ are bounded. It is now an exercise to see that $a\to \beta(h)^{1/2}a\beta(h)^{1/2}$ linearly and homeomorphically maps $L^\psi(\M)$ onto $L_\psi(\M)$. Hence it suffices to prove the claim for the space $L_\psi(\M)$ only. Therefore suppose that the upper Boyd index of $L_\psi(\mathbb{R})$ (denoted by $\overline{\alpha}(L_\psi(\mathbb{R})$) is less than 1. Given that $L^{\psi^*}(\mathbb{R})$ is the K\"othe dual of $L_\psi(\mathbb{R})$, this is equivalent to requiring the lower Boyd index of this dual to be strictly positive, that is $0 < \underline{\alpha}(L^{\psi^*}(\mathbb{R}))$. We will henceforth simply write $\underline{\alpha}$ for $\underline{\alpha}(L^{\psi^*}(\mathbb{R}))$. Now consider the function $$k_{\psi^*}(t) = \sup_{0<s}\frac{\varphi_{\psi^*}(st)}{\varphi_{\psi^*}(s)} = \sup_{0<s}\frac{\varphi_{\psi^*}(s)}{\varphi_{\psi^*}(s/t)}.$$By Lemma 4.8.17 of \cite{BS}, this function is precisely the function defined in equation (8.36) on page 277 of \cite{BS}. It is an exercise to see that this function is non-decreasing and that $k_{\psi^*}(1)=1$. By equation (8.40) on page 279 of \cite{BS} there must exist some $0<\delta<1$ so that $$\left(\frac{t}{2}\right)^{\underline{\alpha}} \leq k_{\psi^*}(t) \leq t^{\underline{\alpha}/2} \quad\mbox{for all}\quad 0<t\leq\delta.$$This estimate together with the fact that  $k_{\psi^*}(t)$ is bounded on $[0,1]$, ensures that $\frac{k_{\psi^*}(t)}{t}$ is integrable on $[0,1]$.

Next observe that for a fixed $t>0$ and any $\lambda>0$, we have that $\frac{\varphi_{\psi^*}(\lambda)}{\varphi_{\psi^*}(\lambda/t)}\leq k_{\psi^*}(t)$. Hence by the Borel functional calculus we have that 
$$d_t =\varphi_{\psi^*}(h)(\varphi_{\psi^*}(\tfrac{1}{t}h))^{-1} \leq k_{\psi^*}(t)\I\quad\mbox{for all}\quad 
0 < t \leq 1.$$In the proof of Theorem 
\ref{mu1} we saw that $t\mu_t(a) = \mu_1(d_t^{1/2}ad_t^{1/2})$ for any $a\in L_{\psi}(\M)$. On applying inequality 
\ref{eq:muineq} to the above facts, it therefore follows that $t\mu_t(a) \leq k_{\psi^*}(t)\mu_1(a)$ for all $0<t\leq 1$. On combining this estimate with the fact that $t\to\mu_t(a)$ is non-increasing, it therefore follows that 
\begin{eqnarray*}
\mu_1(a) &\leq&\int_0^1\mu_t(a)\, dt\\
&\leq& \left(\int_0^1\frac{k_{\psi^*}(t)}{t} \, dt\right) \cdot \sup_{0<s\leq 1}\frac{s}{k_{\psi^*}(s)}\mu_s(a)\\
&=& \left(\int_0^1\frac{k_{\psi^*}(t)}{t} \, dt\right)\mu_1(a).
\end{eqnarray*}
It is known that the quantity $\int_0^1\mu_t(a)\, dt$ is subadditive on $\widetilde{\mathcal{A}}$ \cite{FK}. By the above estimates, this quantity will then provide the required norm on $L_\psi(\M)$.
\end{proof}

\begin{remark}
By Theorem \ref{mu1} we have that $\sup\{t\mu_t(a)| 0<t\leq 1\}=\mu_1(a)$ for any $a$ in either $L^\psi(\M)$ or $L_\psi(\M)$. In fact classically the quantity $\rho(a) = \sup\{t\mu_t(a)| 0<t\leq 1\}$ turns out be a Banach Function quasi-norm on $L^0_+(0,\infty)$ which satisfies the Fatou property. (Although classically it is more common to denote the decreasing rearrangement of an element $f$ by $f^*$, for the sake of uniformity we retain the notation $\mu_t(f)$ for that context as well.) If indeed $\rho(a) = 0$ for some $a \in L^0_+(0,\infty)$, then surely $\mu_t(a) = 0$ for every $0\leq t \leq 1$. Given that $t\to\mu_t(a)$ is non-increasing, we in fact have that $\mu_t(a) =0$ for all $t$, which ensures that $a=0$. To see that the quantity satisfies a generalised triangle inequality, observe that for any $a, b \in L^0_+(0,\infty)$ and any $0<t\leq 1$, we have that
$$t\mu_t(a+b) \leq 2\left(\frac{t}{2}\mu_{t/2}(a)+\frac{t}{2}\mu_{t/2}(b)\right) \leq 2(\rho(a)+\rho(b)).$$Taking the supremum over all such $t$'s yields the fact that $\rho(a+b)\leq2(\rho(a)+\rho(b))$. From the work of Xu \cite{X} we know that the theory of noncommutative rearrangement-invariant 
Banach Function Spaces as developed by Dodds, Dodds, de Pagter et al, extends canonically to include the category of rearrangement-invariant 
quasi-Banach Function Spaces. Thus we may construct the space $L^\rho(\widetilde{\mathcal{A}})$ in the spirit of that theory. What the preceding proposition then tells us, is that all of the Orlicz spaces associated with $\M$ live inside this space, and that their quasi-norms appear as the restriction of the quasi-norm of the superspace $L^\rho(\widetilde{\mathcal{A}})$. In fact using the fact that the fundamental function of the classical space $(L^1+L^\infty)(0,\infty)$ is just $\varphi_{1+\infty}(t)= \min(1,t)$, it is an exercise to show that the space $L^\rho(0,\infty)$ is nothing but the Lorentz-like space $\Lambda_\infty((L^1+L^\infty)(0,\infty))$ of Robert Sharpley (see exercise 21 of chapter 4 of \cite{BS}). Thus in the notation of \cite{DDdP}, we have that $L^\rho(\widetilde{\mathcal{A}}) = \Lambda_\infty(L^1+L^\infty)(\widetilde{\mathcal{A}})$.
\end{remark}

Our next task is to address the issue of K\"othe duality for these Orlicz spaces. To give structure to this theory, we need the following concepts:

\begin{definition}\label{315}
Given an Orlicz function $\psi$, define spaces $S^\psi$ and $S_\psi$ by $$S^\psi = \{a\in\widetilde{\mathcal{A}} : \theta_s(a\widetilde{\varphi}_{\psi^*}(h)^{1/2}\chi_{[0,\delta]}(h)) = e^{-s/2}a\widetilde{\varphi}_{\psi^*}(h)^{1/2}\chi_{[0,e^s\delta]}(h) \mbox{ for all } 0<\delta\}$$and 
$$S_\psi = \{a\in\widetilde{\mathcal{A}} : \theta_s(a{\varphi}_{\psi^*}(h)^{1/2}\chi_{[0,\delta]}(h))  = e^{-s/2}a{\varphi}_{\psi^*}(h)^{1/2}\chi_{[0,e^s\delta]}(h) \mbox{ for all } 0<\delta \}.$$
\end{definition}

The author believes the following relationship between the spaces $L^\psi(\M)$, $L_\psi(\M)$ and $S^\psi$, $S_\psi$ to be valid, but is at present unable to deal with the technical difficulties inherent in such a proof.

\begin{conjecture}\label{genfac}
For any Orlicz function $\psi$ we have that 
$$L^\psi(\M) =\overline{\mathrm{span}}\{a^*b : a, b \in S^\psi\} \qquad\mbox{and}\qquad L_\psi(\M) =\overline{\mathrm{span}}\{a^*b : a, b \in S_\psi\}.$$ 
\end{conjecture}

It is clear from the definition that $a\in S^\psi$ if and only if $|a|^2\in L^\psi(\M)$. (Simmilarly $a\in S_\psi$ if and only if $|a|^2\in L_\psi(\M)$.) Thus for the Orlicz functions $\psi_p$ ($\infty\geq p\geq 1$) discussed in Remark \ref{lp}, it is clear from that discussion that $a\in S^{\psi_p}$ if and only if $|a|^2\in L^p(\M)$. Thus in this case the above conjecture boils down to the well known fact that $L^p(\M)=L^{2p}(\M).L^{2p}(\M)$. If indeed this conjecture proves to be true for all Orlicz spaces, the following result then captures the version of K\"othe duality that seems to be valid in this context. The main difference with more familiar descriptions, is that here one uses the ``factorisability'' of the Orlicz spaces to define K\"othe duality in terms of $L^2$ rather than $L^1$. Once we have more technology at our disposal in the form of Theorem \ref{con2}, we will be able to show that all three conditions mentioned below are equivalent. We will do so at the end of subsection \ref{ss:a}. For the moment the implications shown below are sufficient to meet our immediate needs.

\begin{theorem}\label{K3gen} Consider the following statements:
\begin{enumerate}
\item $a\in S^{\psi}$.
\item $S_{\psi^*} a^* \subset L^2(\M)$.
\item $a\widetilde{\varphi}_{\psi^*}(h)^{1/2}e \in L^2(\M)$ for all projections $e\in \mathfrak{n}_\nu$.
\end{enumerate}
In general we have that $(1)\Rightarrow(2)\Rightarrow(3)$.
\end{theorem}

\begin{proof} Let $a\in S^{\psi}$ and $b \in S_{\psi^*}$ be given and let $\chi_{[\alpha,\beta]}= \chi_{[\alpha,\beta]}(h)$ be the spectral projections from the spectral resolution of $h$. For any $0 <\alpha<\beta<\infty$ and any $s\in \mathbb{R}$, we then have that 
\begin{eqnarray*}
&& \theta_s(a)\chi_{[e^s\alpha,e^s\beta]}\theta_s(b^*)\\ 
&=& \theta_s(a\chi_{[\alpha,\beta]}b^*)\\
&=& \theta_s([a.(\widetilde{\varphi}_{\psi^*}(h)^{1/2}\chi_{[0,\beta]})].(h^{-1/2}\chi_{[\alpha,\infty)}(h)).[b.(\varphi_\psi(h)^{1/2})\chi_{[0,\beta]})]^*\\
&=& e^{-s/2}[a.(\widetilde{\varphi}_{\psi^*}(h)^{1/2}\chi_{[0, e^s\beta]})].(h^{-1/2}\chi_{[e^s\alpha, 0]}).[b.(\varphi_\psi(h)^{1/2}\chi_{[0, e^s\beta]})]^*\\
&=& e^{-s/2}a\chi_{[e^{s}\alpha,e^{s}\beta]}b^*.
\end{eqnarray*}
If now we let $\alpha$ decrease to 0 and $\beta$ increase to $\infty$, it follows that $\theta_s(ab^*) = e^{-s/2}ab^*$ for all $s\in\mathbb{R}$. Hence by the known $L^p$-theory, $ab^* \in L^2(\M)$.

Next suppose that (2) holds. Notice that that for any projection $f \in \mathfrak{n}_\nu$, we have that $[[f\widetilde{\varphi}_{\psi^*}^{1/2}(h)]\varphi_\psi(h)^{1/2}] = [fh^{1/2}] \in L^2(\M)$, and hence that $[f\widetilde{\varphi}_{\psi^*}^{1/2}(h)] \in S_{\psi^*}$. Therefore given $a\in S^{\psi}$, it follows from (2) that $a[\varphi_\psi(h)^{1/2}f] = ([f\varphi_\psi(h)^{1/2}]a^*)^* \in L^2(\M)$. 
\end{proof}

\section{A context for interpolation}

We next give specific attention to the largest and smallest of all Orlicz spaces. In particular we describe and study the type III analogues of the spaces $L^1\cap L^\infty$ and $L^1+L^\infty$, and briefly indicate the consequences of applying the $K$-method of interpolation to these spaces. 
We remind the reader that classically the spaces $L^1\cap L^\infty$ and $L^1+L^\infty$ are both Orlicz spaces respectively corresponding to the Orlicz functions 
\begin{equation}\label{eq:1inf}
\psi_{1\cap\infty}(t) = \left\{\begin{array}{ll} t & \quad 0\leq t\leq 1\\ \infty & \quad t\geq 1\end{array}\right. , \qquad \psi_{1+\infty}(t) = \left\{\begin{array}{ll} 0 & \quad 0\leq t\leq 1\\ t-1 & \quad t\geq 1\end{array}\right..
\end{equation}
Also the fundamental functions for $(L^1\cap L^\infty)(\mathbb{R})$ and $(L^1+L^\infty)(\mathbb{R})$ are respectively given by $\varphi_{1\cap\infty}(t) = \max(1,t)$ and $\varphi_{1+\infty}(t) = \min(1,t)$. It is an exercise to see that canonical norm on $L^1\cap L^\infty$ (see \cite[2.6.1]{BS}) is the Luxemburg norm, whereas it follows from for example \cite[2.6.4]{BS} that the canonical norm on $L^1+L^\infty$ is the Orlicz norm. 

\emph{To simplify notation we adopt the convention of respectively writing $L^{1\cap\infty}$, $\mathfrak{m}_{1\cap\infty}$ and $S^{1\cap\infty}$, for $L^{\psi_{1\cap\infty}}$, $\mathfrak{m}_{\psi_{1\cap\infty}}$ and $S^{\psi_{1\cap\infty}}$. A similar convention is adopted for $L_{\psi_{1+\infty}}$, $\mathfrak{m}_{\psi_{1+\infty}}$ and $S_{\psi_{1+\infty}}$.} 

\medskip

\subsection{The space $L^{1\cap\infty}(\M)$}

\begin{proposition}\label{prop:gen1capinfty}
Consider the space $L^{1\cap\infty}(\M)$. For this space the $*$-subalgebra $\mathfrak{m}_{1\cap\infty}$ of $\M$ consists of all elements $a$ of $\mathfrak{n}_\nu^* \cap \mathfrak{n}_\nu$ for which the strong product $h^{1/2}ah^{1/2}$ is a $\tau_{\mathcal{A}}$-measurable element of $\widetilde{\mathcal{A}}$. Moreover $$L^{1\cap\infty}(\M) = \varphi_{1\cap\infty}(h)^{1/2}[\mathfrak{m}_{1\cap\infty}]\varphi_{1\cap\infty}(h)^{1/2}.$$
\end{proposition}

\begin{proof}
Firstly let $a_0\in\mathfrak{n}_\nu^* \cap \mathfrak{n}_\nu$ be given such that $[h^{1/2}a_0h^{1/2}]$ is a $\tau_{\mathcal{A}}$-measurable element of $\widetilde{\mathcal{A}}$. Both $h^{1/2}a_0$ and $h^{1/2}a_0^*$ (equivalently $h^{1/2}a_0$ and $[a_0h^{1/2}]$) are then of course $\tau_{\mathcal{A}}$-measurable. Next notice that $\varphi_{1\cap\infty}(t) \leq (1+\sqrt{t})^2$ on $[0,\infty)$ with $\frac{\varphi_{1\cap\infty}(t)}{(1+\sqrt{t})^2} \to 1$ as $t\to\infty$. It follows that there exists a constant $K>0$ so that $K(1+\sqrt{t}) \leq \varphi_{1\cap\infty}(t)^{1/2} \leq (1+\sqrt{t})$ on $[0,\infty)$. So then $K(\I+h^{1/2})\leq\varphi_{1\cap\infty}(h)^{1/2} \leq (\I+h^{1/2})$ by the Borel functional calculus. The fact that $\varphi_{1\cap\infty}(h)^{1/2}a_0\varphi_{1\cap\infty}(h)^{1/2}$ is $\tau_{\mathcal{A}}$-measurable, therefore follows from the fact that $[(\I+h^{1/2})a_0(\I+h^{1/2})] = a_0 + [a_0h^{1/2}] + h^{1/2}a_0 + [h^{1/2}a_0h^{1/2}]$ is $\tau_{\mathcal{A}}$-measurable.

Next let $a\in L^{1\cap\infty}(\M)$ be given. In other words suppose we have $a$ with $e\varphi_{1+\infty}(h)^{1/2}a{\varphi}_{1+\infty}(h)^{1/2}f\in L^1(\M)$ for all projections $e,f\in \mathfrak{n}_\nu$. Since $\varphi_{1+\infty}(t)$ is bounded, we in fact have that $\varphi_{1+\infty}(h)^{1/2}\in\mathcal{A}$ and hence that $\varphi_{1+\infty}(h)^{1/2}a{\varphi}_{1+\infty}(h)^{1/2}$ is $\tau_{\mathcal{A}}$-measurable. In this particular case the claim that $e\varphi_{1+\infty}(h)^{1/2}a{\varphi}_{1+\infty}(h)^{1/2}f \in L^1(\M)$ for all projections $e,f \in \mathfrak{n}_\nu$, therefore simplifies to the statement that \newline $\varphi_{1+\infty}(h)^{1/2}a{\varphi}_{1+\infty}(h)^{1/2} \in L^1(\M)$. (To see this notice that by Proposition \ref{orldef} we have that $b_0\varphi_{1+\infty}(h)^{1/2}a{\varphi}_{1+\infty}(h)^{1/2}b_1^* \in L^1(\M)$ for all $eb_0,b_1 \in \mathfrak{n}_\nu$. The claim now follows by considering $f_\alpha\varphi_{1+\infty}(h)^{1/2}a{\varphi}_{1+\infty}(h)^{1/2}f_\beta$ where $\{f_\alpha\}\subset (\mathfrak{n}_\nu)^+$ is a net increasing monotonically to $\I$ \cite[Exercise 2.4.8(d)]{S}.) Notice that $\delta(t) = \frac{\varphi_{1+\infty}(t)}{t} = \frac{\min(1,t)}{t} = \min(1,1/t)$ is also bounded and continuous on $[0,\infty)$. This means that $\delta(h) = h^{-1}\varphi_{1+\infty}(h)$ must be an element of 
$\mathcal{A}$, and hence that $a_0 = \delta(h)^{1/2}a\delta(h)^{1/2} =  h^{-1/2}\varphi_{1+\infty}(h)^{1/2}a\varphi_{1+\infty}(h)^{1/2}h^{-1/2}$ is a well-defined $\tau$-measurable element of $\widetilde{\mathcal{A}}$. For this element $a_0$ we have by construction that $a=\varphi_{1\cap\infty}(h)^{1/2}a_0\varphi_{1\cap\infty}(h)^{1/2}$ and $$h^{1/2}a_0h^{1/2} = \varphi_{1+\infty}(h)^{1/2}a\varphi_{1+\infty}(h)^{1/2} \in L^1(\M).$$ But this means that $e^{-s}h^{1/2}a_0h^{1/2} = \theta_s(h^{1/2}a_0h^{1/2}) = e^{-s}h^{1/2}\theta_s(a_0)h^{1/2}$ for all $s\in\mathbb{R}$. Clearly this can only be if $a_0 = \theta_s(a_0)$ for all 
$s$, in which case $a_0 \in \M$. We have already noted that $h^{1/2}a_0h^{1/2}\in L^1(\M)$. It remains to prove that $a_0\in \mathfrak{n}_\nu^* \cap \mathfrak{n}_\nu$. For this notice that the formal expression $h^{1/2}|a_0|^2h^{1/2}$ may be written as a product of $\tau_{\mathcal{A}}$-measurable operators. Specifically $$h^{1/2}|a_0|^2h^{1/2} =$$ $$\varphi_{1+\infty}(h)^{1/2}[\varphi_{1\cap\infty}(h)^{1/2}a^*\varphi_{1\cap\infty}(h)^{1/2}][h^{-1}\varphi_{1+\infty}(h)] [\varphi_{1\cap\infty}(h)^{1/2}a^*\varphi_{1\cap\infty}(h)^{1/2}]\varphi_{1+\infty}(h)^{1/2}.$$Thus $h^{1/2}|a_0|^2h^{1/2}$ must be $\tau_{\mathcal{A}}$-measurable. Equivalently $a_0\in \mathfrak{n}_\nu$. Since $L^{1\cap\infty}(\M)$ is closed under taking adjoints, we also have $a_0^*\in \mathfrak{n}_\nu$.
\end{proof}

\begin{definition}
We define the norm on $L^{1\cap\infty}(\M)$ to be $$\|\varphi_{1\cap\infty}(h)^{1/2}a\varphi_{1\cap\infty}(h)^{1/2}\|_{1\cap\infty} = \max(\|h^{1/2}ah^{1/2}\|_1, \|h^{1/2}a\|_2, \|ah^{1/2}\|_2, \|a\|_\infty)$$ where $a \in \mathfrak{m}_{1\cap\infty}$.
\end{definition} 

\begin{theorem}\label{thm:genminnorm}
For any $a\in L^{1\cap\infty}(\M)$ we have that $$\frac{1}{16}\mu_1(a) \leq  \|a\|_{1\cap\infty} \leq \mu_1(a).$$The canonical topology on $L^{1\cap\infty}(\M)$ is therefore normable. 
\end{theorem}
 
\begin{proof}
In the proof we make use of the description of $L^{1\cap\infty}(\M)$ given in the preceding proposition. Firstly note that the functions $t\to\frac{1}{\varphi_{1\cap\infty}(t)}$ and $t\to\frac{t}{\varphi_{1\cap\infty}(t)}$ both have 1 as an upper bound. Hence the operators $\frac{1}{\varphi_{1\cap\infty}(h)}$ and  $\frac{h}{\varphi_{1\cap\infty}(h)}$ are both norm 1 elements of $\mathcal{A}$. Given any $a\in\mathfrak{m}_{1\cap\infty}$, it is now a simple exercise to see that each of $\mu_t(a)$, $\mu_t(h^{1/2}a)$, $\mu_t(ah^{1/2})$ and $\mu_t(h^{1/2}ah^{1/2})$ is majorised by $\mu_t(\varphi_{1\cap\infty}(h)^{1/2}a\varphi_{1\cap\infty}(h)^{1/2})$. If now we combine these inequalities with Theorem \ref{thm:uniftop} and the relation of the norm $\|.\|_p$ to $t^{1/p}\mu_t$, it follows that 
\begin{eqnarray*}
\|\varphi_{1\cap\infty}(h)a\|_{1\cap\infty} &=& \max(\|h^{1/2}ah^{1/2}\|_1, \|h^{1/2}a\|_2, \|ah^{1/2}\|_2, \|a\|_\infty)\\
&=& \max(\mu_1(h^{1/2}ah^{1/2}), \mu_1(h^{1/2}a), \mu_1(ah^{1/2}), \mu_1(a))\\
&\leq& \mu_1(\varphi_{1\cap\infty}(h)^{1/2}a\varphi_{1\cap\infty}(h)^{1/2}). 
\end{eqnarray*}

As noted in the proof of the preceding proposition, we may find a constant $K>0$ so that $K(\I+h^{1/2})\leq\varphi_{1\cap\infty}(h)^{1/2} \leq (\I+h^{1/2})$. Therefore 
\begin{eqnarray*}
\mu_{4t}(\varphi_{1\cap\infty}(h)^{1/2}a\varphi_{1\cap\infty}(h)^{1/2}) &=& \mu_{4t}(\varphi_{1\cap\infty}(h)^{1/2}a^*\varphi_{1\cap\infty}(h)a\varphi_{1\cap\infty}(h)^{1/2})^{1/2}\\
&\leq& \mu_{4t}(\varphi_{1\cap\infty}(h)^{1/2}a^*(1+\sqrt{h})^2a\varphi_{1\cap\infty}(h)^{1/2})^{1/2}\\
&=& \mu_{4t}((1+\sqrt{h})a\varphi_{1\cap\infty}(h)a^*(1+\sqrt{h}))^{1/2}\\
&\leq& \mu_{4t}((1+\sqrt{h})a(1+\sqrt{h})^2(h)a^*(1+\sqrt{h}))^{1/2}\\
&=& \mu_{4t}((1+\sqrt{h})a(1+\sqrt{h}))\\
&\leq&  \mu_t(h^{1/2}ah^{1/2}) + \mu_t(h^{1/2}a) + \mu_t(ah^{1/2}) + \mu_t(a).
\end{eqnarray*}
Given $0<\epsilon\leq 1$, we may therefore combine the above inequality with Theorem \ref{thm:uniftop} and the aforementioned relation of the norm $\|.\|_p$ to $t^{1/p}\mu_t$, to see that
\begin{eqnarray*}
&& \epsilon \mu_{4\epsilon}(\varphi_{1\cap\infty}(h)^{1/2}a\varphi_{1\cap\infty}(h)^{1/2})\\ 
&\leq& 4\max(\epsilon\mu_\epsilon(h^{1/2}ah^{1/2}), \epsilon\mu_\epsilon(h^{1/2}a), \epsilon\mu_\epsilon(ah^{1/2}), \epsilon\mu_\epsilon(a))\\
&=& 4\max(\mu_1(h^{1/2}ah^{1/2}), \sqrt{\epsilon}\mu_1(h^{1/2}a), \sqrt{\epsilon}\mu_1(ah^{1/2}), \epsilon\mu_1(a))\\
&\leq& 4\max(\mu_1(h^{1/2}ah^{1/2}), \mu_1(h^{1/2}a), \mu_1(ah^{1/2}), \mu_1(a))\\
&=& 4\|\varphi_{1\cap\infty}(h)^{1/2}a\varphi_{1\cap\infty}(h)^{1/2}\|_{1\cap\infty}.
\end{eqnarray*} 
Setting $\epsilon = \frac{1}{4}$ now yields the inequality $$\frac{1}{16}\mu_1(\varphi_{1\cap\infty}(h)^{1/2}a\varphi_{1\cap\infty}(h)^{1/2}) \leq \|\varphi_{1\cap\infty}(h)^{1/2}a\varphi_{1\cap\infty}(h)^{1/2}\|_{1\cap\infty},$$ and hence proves the theorem.
\end{proof}

\begin{remark}\label{1capinfrem}
By the above theorem the map $$\varphi_{1\cap\infty}^{1/2}(h)a\varphi_{1\cap\infty}(h)^{1/2} \to (a, h^{1/2}a, ah^{1/2}, h^{1/2}ah^{1/2})$$identifies $L^{1\cap\infty}(\M)$ with the subspace 
$K=\{(w, h^{1/2}w, wh^{1/2}, h^{1/2}wh^{1/2}) : w \in \pi(\mathfrak{m}_{1\cap\infty})\}$ of the Banach space $L^\infty(\M) \times L^2(\M;l) \times L^2(\M;r) \times L^1(\M)$ where the spaces $L^2(\M;l)$ and $L^2(\M;r)$ are the ``left'' and ``right'' $L^2$-spaces respectively corresponding to the closures of $\{ah^{1/2} : a\in\pi(\mathfrak{n}_\nu)\}$ and $\{h^{1/2}a : a\in\pi(\mathfrak{n}_\nu)\}$ in $\widetilde{\mathcal{A}}$. Since by Proposition \ref{prop:gen1capinfty} $\mathfrak{m}_{1\cap\infty}$ corresponds to the largest subspace of $\M$ which canonically embeds into each of each $L^2(\M;l)$, $L^2(\M;r)$, and $L^1(\M)$, it follows that $L^{1\cap\infty}(\M)$ may be regarded as some sort of intersection of the spaces $L^\infty(\M)$, $L^2(\M;r)$, $L^2(\M;l)$ and $L^1(\M)$.
\end{remark}

\medskip

\subsection{The space $L_{1+\infty}(\M)$}\label{ss:a} 
\~
\medskip

The first thing we note about the Orlicz space $L_{1+\infty}(\M)$, is that it is large enough to accommodate all the other Orlicz spaces.

\begin{remark}\label{rem:gen1+infty}
Let $\psi$ be a general Orlicz function. With respect to the topology of convergence in measure, the Orlicz spaces $L^\psi(\M)$ and $L_\psi(\M)$ inject continuously into $L_{1+\infty}(\M)$. The other case being similar, we explain how this works for $L^\psi(\M)$. Since $t\to \widetilde{\varphi}_{\psi^*}(t)$ is continuous and increasing on $(0,\infty)$ whereas  $t\to \frac{\widetilde{\varphi}_{\psi^*}(t)}{t}$ is continuous and decreasing, the function $t \to  \frac{\widetilde{\varphi}_{\psi^*}(t)}{\varphi_{1\cap\infty}(t)}$ is continuous and bounded on 
$(0,\infty)$. For simplicity of notation write $\zeta_\psi(t) = \frac{\widetilde{\varphi}_{\psi^*}(t)}{\varphi_{1\cap\infty}(t)}$. The boundedness of 
$\zeta_\psi$ ensures that $\zeta_\psi(h) \in \mathcal{A}$. Given $a\in L^\psi(\M)$, it is now clear that 
$e\varphi_{1\cap\infty}(h)^{1/2}.\zeta_\psi(h)^{1/2}a\zeta_\psi(h)^{1/2}.\varphi_{1\cap\infty}(h)^{1/2}f = e\widetilde{\varphi}_{\psi^*}(h)^{1/2}a\widetilde{\varphi}_{\psi^*}(h)^{1/2}f \in L^1(\M)$ for all projections $e,f \in \pi(\mathfrak{n}_\nu)$, 
and hence that $\zeta_\psi(h)^{1/2}a\zeta_\psi(h)^{1/2} \in L_{1+\infty}(\M)$. This clearly means that the formal prescription 
$b \to \zeta_\psi(h)^{1/2}b\zeta_\psi(h)^{1/2}$ yields a well-defined map from $L^\psi(\M)$ to $L_{1+\infty}(\M)$. We shall denote this embedding by $\iota_\psi$. The claim regarding the continuity of the embedding follows from the fact that $\mu_t(\iota_\psi(a)) = \mu_t(\zeta_\psi(h)^{1/2}a\zeta_\psi(h)^{1/2}) \leq \|\zeta_\psi(h)\|_\infty\mu_t(a)$ for all $t \geq 0$.  
\end{remark}

We saw in Remark \ref{1capinfrem} that $L^{1\cap\infty}(\M)$ may be viewed as some sort of intersection of $L^1$, $L^\infty$ and the left and right $L^2$ spaces $L^2(\M;l)$ and $L^2(\M;r)$. Given the nature of $L^{1\cap\infty}(\M)$ and the classical duality between $L^{1\cap\infty}$ and $L_{1+\infty}$, we expect that by contrast $L_{1+\infty}(\M)$ should somehow be related to the sum of these four spaces. In the remainder of this subsection we will labour toward giving concrete expression to this suspicion. To describe how this works we shall make use of the embeddings of 
$L^1$ and $L^\infty$ into $L^{1+\infty}(\M)$ as defined above, and also similarly define embeddings of $L^2(\M;l)$ and $L^2(\M;r)$ into $L_{1+\infty}(\M)$. With $\zeta_1(t)$ denoting $\frac{\varphi_{1+\infty}(t)}{t}$, these embeddings effectively work as follows:
\begin{eqnarray*}
\iota_\infty: L^\infty(\M) \to L_{1+\infty}(\M)&:& a \to \varphi_{1+\infty}(h)^{1/2}a\varphi_{1+\infty}(h)^{1/2}\\
\iota_{2,r}: L^2(\M;r) \to L_{1+\infty}(\M)&:& a \to \zeta_1(h)^{1/2}a\varphi_{1+\infty}(h)^{1/2}\\
\iota_{2,l}: L^2(\M;l) \to L_{1+\infty}(\M)&:& a \to \varphi_{1+\infty}(h)^{1/2}a\zeta_1(h)^{1/2}\\
\iota_1: L^1(\M) \to L_{1+\infty}(\M)&:& a \to \zeta_1(h)^{1/2}a\zeta_1(h)^{1/2}\\
\end{eqnarray*}
An easy modification of the arguments in the preceding remark shows that $\zeta_1(h)$ is bounded and the embeddings well-defined.

\emph{In the following we will write $X_{1+\infty}$ for $\mathrm{span}[(S_{1+\infty})^*S_{1+\infty}]$}.
  
\begin{definition}
Suppose that the equalities described in the Conjecture \ref{genfac} hold for the pair $(\M, \nu)$. Given any Orlicz function $\psi$, the subspace $\mathrm{span}\{(S^{\psi})^*S^{\psi}\}$ of $L^\psi(\M)$ has a dual action on $L_\psi(\M)$ defined as follows: Given $a_0,a_1 \in S^\psi$ and $b_0, b_1 \in S_{\psi^*}$ we define the quantity $\wtr$ by $$\wtr((a_0^*a_1)(b_0^*b_1)) = \tr(a_1b_0^*b_1a_0^*) = \tr(b_1a_0^*a_1b_0^*).$$(Here we make use of the fact that both $S^\psi(S_{\psi^*})^*$ and $S_{\psi^*}(S^\psi)^*$ live inside $L^2(\M)$.)By linearity this action extends to an action of $\mathrm{span}\{(S^{\psi})^*S^{\psi}\}$ on $\mathrm{span}\{(S_{\psi^*})^*S_{\psi^*}\}$. Since for any $a_0,a_1 \in S^\psi$ and $b\in\mathrm{span}\{(S_{\psi^*})^*S_{\psi^*}\}$ we have that $$|\wtr((a_0^*a_1)b)| = |\tr(a_1ba_0^*)| = 3\mu_3(a_1ba_0^*)\leq 3\mu_1(a_0)\mu_1(a_1)\mu_1(b),$$it is clear that the action of $\mathrm{span}\{(S^{\psi})^*S^{\psi}\}$ on $\mathrm{span}\{(S_{\psi^*})^*S_{\psi^*}\}$ is continuous and hence extends by continuity to all of $L_{\psi^*}(\M)=\overline{\mathrm{span}}\{(S_{\psi^*})^*S_{\psi^*}\}$. Similarly $\mathrm{span}\{(S_{\psi^*})^*S_{\psi^*}\}$ induces a continuous action on $L^\psi(\M)$.
\end{definition}

\begin{theorem}\label{con2}
$X_{1+\infty} = \iota_1(L^1(\M)) + \iota_{2,r}(L^2(\M;r)) + \iota_{2,l}(L^2(\M;l)) + \iota_\infty(L^\infty(\M))$.
\end{theorem}

\begin{proof}
It is an exercise to see that $X_{1+\infty} \supset \iota_1(L^1(\M)) + \iota_{2,r}(L^2(\M;r)) + \iota_{2,l}(L^2(\M;l)) + \iota_\infty(L^\infty(\M))$. It therefore remains to prove the converse. We firstly note that each element of $X_{1+\infty}$ induces a well defined continuous linear functional on the subspace $\varphi_{1\cap\infty}(h)^{1/2}\mathfrak{m}_\nu\varphi_{1\cap\infty}(h)^{1/2}$ of $L^{1\cap\infty}(\M)$. To see this let $a \in X_{1+\infty}$, $b\in \mathfrak{m}_\nu$ be given with $a=a_0^*a_1$ and $b=b_0^*b_1$ where $a_0,a_1 \in S_{1+\infty}$ and $b_0, b_1 \in \mathfrak{n}_\nu$. The action of $X_{1+\infty}$ on $\varphi_{1\cap\infty}\mathfrak{m}_\nu\varphi_{1\cap\infty}$ is then defined by setting
$$\wtr(ab) = \tr((b_1\varphi_{1\cap\infty}(h)^{1/2}a_0^*)(a_1\varphi_{1\cap\infty}(h)^{1/2}b_0^*)).$$To see that this quantity is well-defined notice that by construction both $a_0\varphi_{1\cap\infty}(h)^{1/2}$ and $a_1\varphi_{1\cap\infty}(h)^{1/2}$ belong to $S^{1\cap\infty}$, which then ensures that $(b_1\varphi_{1\cap\infty}(h)^{1/2}a_0^*)$ and $(a_1\varphi_{1\cap\infty}(h)^{1/2}b_0^*)$ belong to $L^2(\M)$. Uniqueness follows from the fact that $$\tr((b_1\varphi_{1\cap\infty}(h)^{1/2}a_0^*)(a_1\varphi_{1\cap\infty}(h)^{1/2}b_0^*)) = \tr((a_1\varphi_{1\cap\infty}(h)^{1/2}b_0^*)(b_1\varphi_{1\cap\infty}(h)^{1/2}a_0^*)).$$
On setting $w=\varphi_{1\cap\infty}(h)^{1/2}b_0^*b_1\varphi_{1\cap\infty}(h)^{1/2}$ for ease of notation, it will now additionally follow from the relation of $\|.\|_1$ and $t\mu_t$, and Theorem \ref{thm:genminnorm} that 
\begin{eqnarray*}
|\wtr(wa)| = |\tr(a_0wa_1^*)| &\leq& \|a_0wa_1^*\|_1\\
&=& 3\mu_3(a_0wa_1^*)\\
&\leq& 3\mu_1(a_0)\mu_1(a_1)\mu_1(w)\\
&\leq& 48\mu_1(a_0)\mu_1(a_1)\|w\|_{1\cap\infty}.
\end{eqnarray*}
It follows that $w \to \wtr(aw)$ is a continuous linear functional $F_a$ on $\varphi_{1\cap\infty}(h)^{1/2}\mathfrak{m}_\nu\varphi_{1\cap\infty}(h)^{1/2}$, with norm $\|F_a\| \leq 48\mu_1(a_0)\mu_1(a_1)$. 

As we saw in Remark \ref{1capinfrem}, $L^{1\cap\infty}(\M)$ may be identified with the subspace 
$K=\{(w, h^{1/2}w, wh^{1/2}, h^{1/2}wh^{1/2}) : w \in \pi(\mathfrak{m}_{1\cap\infty})\}$ of the Banach space $L^\infty(\M) \times L^2(\M;l) \times L^2(\M;r) \times L^1(\M)$ equipped with the norm $$\|(c,d,e,f)\| = \max(\|c\|_\infty, \|d\|_2, \|e\|_2, \|f\|_1)$$by means of the map $\varphi_{1\cap\infty}(h)^{1/2}w\varphi_{1\cap\infty}(h)^{1/2} \to (w, h^{1/2}w, wh^{1/2}, h^{1/2}wh^{1/2})$. The Banach dual of this space is 
$L^\infty(\M)^*\times L^2(\M;r) \times L^2(\M;l) \times L^\infty(\M)$ equipped with the norm $\|(c,d,e,f)\| = \|c\|+\|d\|+\|e\|+\|f\|$. The dual of $K$ may in turn be realised as the quotient space $[L^\infty(\M)^*\times L^2(\M;r) \times L^2(\M;l) \times L^\infty(\M)]/K^\circ$ (where $K^\circ$ is the polar of $K$). For each $(c,d,e,f)$ in some equivalence class of this quotient space, the corresponding norm may be realised by $\|[(c,d,e,f)]\| = \inf\|g\|$ where the infimum is taken over all elements $g$ of $L^\infty(\M)^*\times L^2(\M;r) \times L^2(\M;l) \times L^\infty(\M)$ which agree with $(c,d,e,f)$ on $K$. By the Hahn-Banach theorem, the functional $F_a$ can with preservation of norm be extended to a functional acting on all of $L^\infty(\M) \times L^2(\M;l) \times L^2(\M;r) \times L^1(\M)$. Let this functional be $(g_1, g_{2,r}, g_{2,l}, g_\infty)$. Now of course in its action on $L^1(\M)$, $g_\infty$ is of the form $g_\infty(b) = \tr(a_\infty b)$ for some $a_\infty \in L^\infty(\M)$. As far as the actions of $g_{2,r}$ and $g_{2,l}$ on $L^2(\M;l)$ and $L^2(\M;r)$ are concerned, we may similarly find elements $a_{2,r}\in L^2(\M;r)$ and $g_{2,l} \in L^2(\M;l)$ so that $g_{2,r}(\cdot) = \tr(\cdot a_{2,r})$ and $g_{2,l}=\tr(\cdot a_{2,l})$. The challenge now is to show that $a - [\iota_{2,r}(a_{2,r}) + \iota_{2,l}(a_{2,l}) + \iota_\infty(a_\infty)]$ is of the form $\iota_1(a_1)$ for some $a_1\in L^1(\M)$, and that the action of $g_1$ on $L^\infty(\M)$ is induced by $a_1$. To simplify notation we will write $a_0 = \iota_{2,r}(a_{2,r}) + \iota_{2,l}(a_{2,l}) + \iota_\infty(a_\infty)$. 

Since the extension $(g_1, g_{2,r}, g_{2,l}, g_\infty)$ of $F_a$ preserves the norm of $F_a$, we have that $\|g_1\| + \|g_{2,r}\| + \|g_{2,l}\| + \|g_\infty\| = \|F_a\| = \inf\|g\|$ where as before the infimum is taken over all elements $g$ of $L^\infty(\M)^*\times L^2(\M;r) \times L^2(\M;l) \times L^\infty(\M)$ which agree with $(g_1, g_{2,r}, g_{2,l}, g_\infty)$ on $K$. Notice that $(0, g_{2,r}, g_{2,l}, g_\infty)$ is then an extension of $F_{a_0}$. We claim that $\|(0, g_{2,r}, g_{2,l}, g_\infty)\|= \|g_{2,r}\| + \|g_{2,l}\| + \|g_\infty\| = \|F_{a_0}\| = \inf\{\|g\|: g \mbox{ an extension of } F_{a_0}\}$, and hence that $(0, g_{2,r}, g_{2,l}, g_\infty)$ is the Hahn-Banach extension of $F_{a_0}$. If this norm equality did not hold, we would be able to find an extension $(f_1, f_{2,r}, f_{2,l}, f_\infty)$ of $\|F_{a_0}\|$ with $\|(f_1, f_{2,r}, f_{2,l}, f_\infty)\| < \|(0, g_{2,r}, g_{2,l}, g_\infty)\| = \|g_{2,r}\| + \|g_{2,l}\| + \|g_\infty\|$. The functional $(g_1+f_1, f_{2,r}, f_{2,l}, f_\infty)$ would then be an extension of $F_a$ for which $\|(g_1+f_1, f_{2,r}, f_{2,l}, f_\infty)\| < \|g_1\| + \|g_{2,r}\| + \|g_{2,l}\| + \|g_\infty\| = \|F_a\| = \inf\{\|g\| : g \mbox{ an extension of } F_{a}\}$, which is clearly impossible. Since $(g_1, g_{2,r}, g_{2,l}, g_\infty)$ is the Hahn-Banach extension of $F_a$ and $(0, g_{2,r}, g_{2,l}, g_\infty)$ the Hahn-Banach extension of $F_{a_0}$, $(g_1,0,0,0)$ is an extension of $F_{a-a_0}$. So for any $b_0,b_1\in\mathfrak{n}_\nu$ we have that 
\begin{eqnarray}\label{ineq:aa}
|\tr(b_0\varphi_{1\cap\infty}(h)^{1/2}(a&-&a_0)\varphi_{1\cap\infty}(h)^{1/2}b_1^*)|\nonumber\\
&=& |\wtr(\varphi_{1\cap\infty}(h)^{1/2}b_1^*b_0\varphi_{1\cap\infty}(h)^{1/2}(a-a_0))|\nonumber\\
&=& |<(g_1, 0, 0, 0),(b_1^*b_0, b_1^*b_0h^{1/2}, h^{1/2}b_1^*b_0, h^{1/2}b_1^*b_0h^{1/2})>|\\
&=& |g_1(b_1^*b_0)|\nonumber\\
&\leq& \|g_1\|.\|b_1^*b_0\|_\infty.\nonumber
\end{eqnarray}

We are now ready to show that $\varphi_{1\cap\infty}(h)^{1/2}(a-a_0)\varphi_{1\cap\infty}(h)^{1/2}$ belongs to $L^1(\M)$. However this will be done in several stages. First we gather some information regarding the extended positive part of $\mathcal{A}$.  Recall that any positive operator $k$ affiliated with $\mathcal{A}$ for which $\theta_s(k)=e^{-s}k$, corresponds uniquely to a normal semifinite weight $\phi_k$ on $\M$ with $k$ appearing as the density $\frac{d\widetilde{\phi_k}}{d\tau_{\mathcal{A}}}$ of the dual weight $\widetilde{\phi_k}$ with respect to the trace $\tau_{\mathcal{A}}$ (see chapter II of \cite{Tp}). Now let $T$ be the canonical operator-valued weight from the extended positive part of $\mathcal{A}$ to the extended positive part of $\M$. Select $f$ in the extended positive part of $\mathcal{A}$ so that $T(f)=\I$. Then in the notation of Proposition 1.11 of \cite{Ha}, we have that $\phi_k(\I) = \widetilde{\phi_k}(f) = \tau_{\mathcal{A}}(k\cdot f)$. If now $\{k_\alpha\}$ was a net of positive affiliated operators increasing pointwise to $k$, then by \cite[Proposition 1.11]{Ha} we would have $\sup_\alpha\phi_{k_\alpha}(\I) = \sup_\alpha\tau_{\mathcal{A}}(k_\alpha\cdot f) = \tau_{\mathcal{A}}(k\cdot f) = \phi_k(\I)$. 

By Lemma \ref{tech2}, there exists a net $\{p_\alpha\}$ in $(\mathfrak{n}_\nu)^+$ for which $\{p_\alpha^2\}$ increases monotonically (and hence $\sigma$-strongly) to $\I$. Fixing $\beta$ for the moment, the fact that $[\varphi_{1\cap\infty}(h)^{1/2}p_\alpha][p_\alpha\varphi_{1\cap\infty}(h)^{1/2}]^2$ increases monotonically to $\varphi_{1\cap\infty}(h)$, ensures that the operators $|[p_\alpha\varphi_{1\cap\infty}(h)^{1/2}](a-a_0)[\varphi_{1\cap\infty}(h)^{1/2}p_\beta]|^2$ increase to the (possibly not densely defined) operator $|[\varphi_{1\cap\infty}(h)^{1/2}(a-a_0)[\varphi_{1\cap\infty}(h)^{1/2}p_\beta]]|^2$, as $p_\alpha^2\to\I$.
Passing to square roots, it follows that the operators $k_{\alpha\beta}=|[p_\alpha\varphi_{1\cap\infty}(h)^{1/2}](a-a_0)[\varphi_{1\cap\infty}(h)^{1/2}p_\beta]| \in L^1(\M)$ increase to $k_\beta = |\varphi_{1\cap\infty}(h)^{1/2}(a-a_0)[\varphi_{1\cap\infty}(h)^{1/2}p_\beta]|$ as $p_\alpha\to\I$. A priori it is of course not at all clear that $k_\beta$ exists as a densely defined affiliated operator. However recall that the extended positive part of $\mathcal{A}$ is closed under increasing pointwise limits, and hence that we may give meaning to this latter object as an element of the extended positive part of $\mathcal{A}$. 

Now let $\phi_{\alpha\beta}$ denote the finite normal weight on $\M$ corresponding to $k_{\alpha\beta}$. (Since $k_{\alpha\beta} \in L^1(\M)$, we actually have that $\phi_{\alpha\beta}\in(\M_*)^+$.) We will use these weights to compute the normal weight associated with $k_\beta$ and show that it belongs to $L^1(\M)$.  Notice that since each $\phi_{\alpha\beta}$ belongs to $(\M_*)^+$, the expression $\phi_0 = \sup_\alpha \phi_{\alpha\beta}$ defines a normal weight on $(\M)^+$. Let $v_{\alpha\beta}$ be the partial isometries in the polar decomposition of $[p_\alpha\varphi_{1\cap\infty}(h)^{1/2}](a-a_0)[\varphi_{1\cap\infty}(h)^{1/2}p_\beta]$. Since for each $\alpha$ we have that 
\begin{eqnarray*}
\phi_{\alpha\beta}(\I) &=& \tr(k_{\alpha\beta})\\
&=& \tr(v^*_{\alpha\beta}[p_\alpha\varphi_{1\cap\infty}(h)^{1/2}](a-a_0)[\varphi_{1\cap\infty}(h)^{1/2}p_\beta])\\
&=& |g_1(p_\beta v^*_{\alpha\beta}p_\alpha)|\\
&\leq& \|g_1\|,
\end{eqnarray*}
it is clear that $\phi_0(\I)\leq \|g_1\| <\infty$. In other words $\phi_0 \in (\M_*)^+$. Passing to the dual weights and denoting the normal weight on $\mathcal{A}$ induced by $k_\beta$ by $\nu_\beta$, we will for any $g$ in the extended positive part of $\mathcal{A}$ have that $$\widetilde{\phi_0}(g) = \sup_\alpha\widetilde{\phi_{\alpha\beta}}(g) = \sup_\alpha\tau_{\mathcal{A}}(k_{\alpha\beta}\cdot g) = \tau_{\mathcal{A}}(k_\beta\cdot g) = \nu_\beta(g)$$(here we once again used \cite[Theorem 1.12]{Ha}). Clearly $\widetilde{\phi_0}$ must then agree with $\nu_\beta$.

Now since $\widetilde{\phi_0}$ agrees with $\nu_\beta$, this in turn ensures that the element $k_0$ in $L^1(\M)$ corresponding to $\phi_0\in\M_*$, agrees with $k_\beta = |\varphi_{1\cap\infty}(h)^{1/2}(a-a_0)[\varphi_{1\cap\infty}(h)^{1/2}p_\beta]|$. Thus $k_\beta=|\varphi_{1\cap\infty}(h)^{1/2}(a-a_0)[\varphi_{1\cap\infty}(h)^{1/2}p_\beta]|$ is not only an operator, but in fact an element of $L^1(\M)$. 
Equivalently we have that $\varphi_{1\cap\infty}(h)^{1/2}(a-a_0)[\varphi_{1\cap\infty}(h)^{1/2}p_\beta]$ is closable with minimal closed extension an element of $L^1(\M)$. Taking adjoints we therefore have that  
$[[p_\beta\varphi_{1\cap\infty}(h)^{1/2}](a^*-a^*_0)\varphi_{1\cap\infty}(h)^{1/2}] \in L^1(\M)$ for every $\beta$. On repeating essentially the same argument as before, we can now show that the operators $[[p_\beta\varphi_{1\cap\infty}(h)^{1/2}](a^*-a^*_0)\varphi_{1\cap\infty}(h)^{1/2}]$ increase to  $[\varphi_{1\cap\infty}(h)^{1/2}(a^*-a^*_0)\varphi_{1\cap\infty}(h)^{1/2}] \in L^1(\M)$, and hence that $w=[\varphi_{1\cap\infty}(h)^{1/2}(a-a_0)\varphi_{1\cap\infty}(h)^{1/2}] \in L^1(\M)$.

It is clear from \ref{ineq:aa} above that by construction $\tr(bw) = g_1(b)$ for any $b\in \mathfrak{m}_\nu$. This in turn ensures that $\|w\|_1 \leq \|g_1\|$. If we write $g_w$ for the functional $\tr(w\cdot)$, it follows that $(g_w, g_{2,r}, g_{2,l}, g_\infty)$ is an extension of $F_a$ which must be a Hahn-Banach extension, since $\|(g_w, g_{2,r}, g_{2,l}, g_\infty)\| \leq \|g_1\| + \|g_{2,r}\| + \|g_{2,l}\| + \|g_\infty\| = \|F_a\| = \inf\{\|g\| : g \mbox{ an extension of } F_{a}\}$. We may therefore without loss of generality replace $g_1$ with $g_w$ if necessary. This ensures that as far as its action on $K$ is concerned, $(g_1, g_{2,r}, g_{2,l}, g_\infty)$ is indeed induced by $\iota_1(w)+\iota_{2,r}(a_{2,r}) + \iota_{2,l}(a_{2,l}) + \iota_\infty(a_\infty)$. This in turn is sufficient to ensure that $a = \iota_1(w)+\iota_{2,r}(a_{2,r}) + \iota_{2,l}(a_{2,l}) + \iota_\infty(a_\infty)$.
\end{proof}

\begin{definition}
We define the norm on $X_{1+\infty}$ to be $$\|b\|_{1+\infty} = \inf(\|c\|_1 + \|d\|_2 + \|e\|_2 + \|f\|_\infty)$$where the infimum is taken over all representations of $b$ of the form $b = \iota_1(c) + \iota_{2,r}(d) + \iota_{2,l}(e) + \iota_\infty(f)$ where $c\in L^1(\M)$, $d\in L^2(\M;r)$, $e\in L^2(\M;l)$ and $d\in L^\infty(\M)$.
\end{definition}

\begin{remark}
We observe that the infimum in the definition of the norm of $X_{1+\infty}(\M)$ is actually a minimum. In the notation of the proof, this much can be deduced from the fact that $(g_w, g_{2,r}, g_{2,l}, g_\infty)$ is a Hahn-Banach extension of $F_a$.
\end{remark}

\begin{corollary}\label{dualpair}
The spaces $X_{1+\infty}$ and $L^{1\cap\infty}(\M)$ form a dual pair with the dual action of $X_{1+\infty}$ on $L^{1\cap\infty}(\M)$  defined by $$\langle a,b\rangle = \tr(b_1a_0 + b_{2,r}(h^{1/2}a_0) + b_{2,l}(a_0h^{1/2}) + b_\infty (h^{1/2}a_0h^{1/2}))$$where $a\in L^{1\cap\infty}(\M)$, $b\in X_{1+\infty}$, where $a_0 \in \mathfrak{m}_{1\cap\infty}$ is selected so that \newline $\varphi_{1\cap\infty}(h)^{1/2}a_0\varphi_{1\cap\infty}(h)^{1/2}=a$, and where $b_1\in L^1(\M)$, $b_{2,r}\in L^2(\M;r)$, $b_{2,l}\in L^2(\M;l)$ and $b_\infty\in L^\infty(\M)$ are selected so that $b = \iota_1(b_1) + \iota_{2,r}(b_{2,r}) + \iota_{2,l}(b_{2,l}) + \iota_\infty(b_\infty)$. Moreover $$|\langle a,b\rangle| \leq \|b\|_{1+\infty}.\|a\|_{1\cap\infty}.$$ 
\end{corollary}

\begin{proof}
Let $a\in L^{1\cap\infty}(\M)$, $b\in X_{1+\infty}$ be given as in the hypothesis. We show that $\langle a,b\rangle$ is uniquely defined. If indeed $b_\infty$ was of the form $b_1 = c+d$ where $d$ was an element of $\M$ for which $h^{1/2}d \in \widetilde{\mathcal{A}}$, it is then a simple matter to conclude that
\begin{eqnarray*}
&& \tr(b_1a_0 + b_{2,r}(h^{1/2}a_0) + b_{2,l}(a_0h^{1/2}) + b_\infty (h^{1/2}a_0h^{1/2}))\\ 
&=& \tr(b_1a_0 + b_{2,r}(h^{1/2}a_0) + b_{2,l}(a_0h^{1/2}) + [c(h^{1/2}a_0h^{1/2}) + dh^{1/2}a_0h^{1/2}])\\
&=& \tr(b_1a_0 + b_{2,r}(h^{1/2}a_0) + b_{2,l}(a_0h^{1/2}) + [c(h^{1/2}a_0h^{1/2}) + h^{1/2}dh^{1/2}a_0])\\
&=& \tr(b_1a_0 + b_{2,r}(h^{1/2}a_0) + [b_{2,l}+h^{1/2}d](a_0h^{1/2}) + c(h^{1/2}a_0h^{1/2}))\\
\end{eqnarray*}
Similar observations hold for the cases where either one of $dh^{1/2}$ or $h^{1/2}dh^{1/2}$ belong to $\widetilde{\mathcal{A}}$. On repeating the same sort of argument for $b_1$, $b_{2,r}$ and $b_{2,l}$, it becomes clear that each representation of $b$ of the form $b =  \iota_1(b_1) + \iota_{2,r}(b_{2,r}) + \iota_{2,l}(b_{2,l}) + \iota_\infty(b_\infty)$, will yield the same value for $\tr(b_1a_0 + b_{2,r}(h^{1/2}a_0) + b_{2,l}(a_0h^{1/2}) + b_\infty (h^{1/2}a_0h^{1/2}))$. Thus $\langle a,b\rangle$ is well defined.

It remains to prove that $|\langle a,b\rangle| \leq \|b\|_{1+\infty}.\|a\|_{1\cap\infty}$. Suppose we have a representation of $b$ of the form $b =  \iota_1(b_1) + \iota_{2,r}(b_{2,r}) + \iota_{2,l}(b_{2,l}) + \iota_\infty(b_\infty)$ with $a_0 \in \mathfrak{m}_{1\cap\infty}$ selected so that $a=\varphi_{1\cap\infty}(h)^{1/2}a_0\varphi_{1\cap\infty}(h)^{1/2}$. Then 
\begin{eqnarray*}
&& |\langle a,b\rangle|\\
&=& |\tr(b_1a_0 + b_{2,r}(h^{1/2}a_0) + b_{2,l}(a_0h^{1/2}) + b_\infty (h^{1/2}a_0h^{1/2}))|\\
&\leq& |\tr(b_1a_0)| + |\tr(b_{2,r}(h^{1/2}a_0))| + |\tr(b_{2,l}(a_0h^{1/2}))| + |\tr(b_\infty (h^{1/2}a_0h^{1/2}))|\\
&\leq& \|{b_1}\|_1.\|a_0\|_\infty + \|b_{2,r}\|_2.\|h^{1/2}a_0\|_2+ \|b_{2,l}\|_2.\|a_0h^{1/2}\|_2 + \|b_\infty\|_\infty.\|h^{1/2}a_0h^{1/2}\|_1\\
&\leq& \max(\|a_0\|_\infty, \|h^{1/2}a_0\|_2, \|a_0h^{1/2}\|_2, \|h^{1/2}a_0h^{1/2}\|_1).(\|{b_1}\|_1 + \|b_{2,r}\|_2+ \|b_{2,l}\|_2 + \|b_\infty\|_\infty)\\
&=& \|a\|_{1\cap\infty}.(\|{b_1}\|_1 + \|b_{2,r}\|_2+ \|b_{2,l}\|_2 + \|b_\infty\|_\infty).
\end{eqnarray*}
The result now follows by taking the infimum over all representations of $b$ of the form $b =  \iota_1(b_1) + \iota_{2,r}(b_{2,r}) + \iota_{2,l}(b_{2,l}) + \iota_\infty(b_\infty)$ with $a_0 \in \mathfrak{m}_{1\cap\infty}$.
\end{proof}

\begin{lemma}\label{lem:genmaxnorm}
Let $0<\epsilon\leq 4$ be given. For any $a\in X_{1+\infty}$ we have that $$\frac{1}{4}\epsilon\mu_\epsilon(a) \leq  \|a\|_{1+\infty}.$$
\end{lemma}

\begin{proof}
Let $a \in X_{1+\infty}(\M)$ be given and select $c\in L^1(\M)$, $d\in L^2(\M;r)$, $e\in L^2(\M;l)$ and $f\in L^\infty(\M)$ with $a = \iota_1(c) + \iota_{2,r}(d) + \iota_{2,l}(e) + \iota_\infty(f)$. Given any Orlicz function $\psi$, we noted in Remark \ref{rem:gen1+infty} that $\zeta_\psi(t)$ is bounded. For the specific Orlicz function $\psi(t)=t$, we in fact have that $\zeta_\psi(t) = \zeta_1(t)$ is bounded above by 1. Since $\varphi_{1+\infty}(t)$ is similarly bounded above by 1, the Borel functional calculus ensures that $\zeta_1(h)$ and $\varphi_{1+\infty}(h)$ are contractive elements of $\mathcal{A}$. For $0 < \epsilon\leq 1$, we may therefore conclude from the discussion immediately following Remark \ref{rem:gen1+infty}, that 
$$\mu_{4\epsilon}(a) \leq \mu_{\epsilon}(c) + \mu_{\epsilon}(d) + \mu_{\epsilon}(e) + \mu_{\epsilon}(f)).$$ Therefore on applying Theorem \ref{thm:uniftop}, we may conclude that $$\epsilon\mu_{4\epsilon}(a) \leq \|c\|_1 + \sqrt{\epsilon}\|d\|_2 + \sqrt{\epsilon}\|e\|_2 + \epsilon\|f\|_\infty \leq  \|c\|_1 + \|d\|_2 + \|e\|_2 + \|f\|_\infty.$$Now take the infimum over all representations of $a$ of the form $a = \iota_1(c) + \iota_{2,r}(d) + \iota_{2,l}(e) + \iota_\infty(f)$ to see that $$\epsilon\mu_{4\epsilon}(a)\leq \|a\|_{1+\infty}.$$
\end{proof}

\begin{theorem}\label{412}
The space $X_{1+\infty}$ is complete with respect to the norm $\|\cdot\|_{1+\infty}$. 
\end{theorem}

\begin{proof}
We have already seen in Theorem \ref{con2} that under the norm $\|\cdot\|_{1+\infty}$, the space $X_{1+\infty}$ may be realised as a 
subspace of the Banach dual of $L^{1\cap\infty}(\M)$. Let $\{b_n\}$ be a Cauchy sequence in $X_{1+\infty}$ for which the induced 
functionals $F_{a_n}$ converge to some functional $F$ on $L^{1\cap\infty}(\M)$. By the preceding lemma, any sequence which is Cauchy in 
the norm $\|\cdot\|_{1+\infty}$, must also be Cauchy in the quasi-norm $\mu_1(\cdot)$ on $L_{1+\infty}(\M)$. By Proposition \ref{qntop} 
this means that the sequence $\{b_n\}$ must converge in measure to some element $b\in L_{1+\infty}(\M)$. For any pair of elements $a_0, a_1 \in \mathfrak{n}_\nu$, $a_0\varphi_{1\cap\infty}(h)^{1/2}b_n\varphi_{1\cap\infty}(h)^{1/2}a_1^*$ must then also converge in measure 
to $a_0\varphi_{1\cap\infty}(h)^{1/2}b\varphi_{1\cap\infty}(h)^{1/2}a_1^*$. Consequently 
\begin{eqnarray*}
F(\varphi_{1\cap\infty}(h)^{1/2}a_1^*a_0\varphi_{1\cap\infty}(h)^{1/2}) &=& \lim_{n\to\infty} F_{b_n}(\varphi_{1\cap\infty}(h)^{1/2}a_1^*a_0\varphi_{1\cap\infty}(h)^{1/2})\\
&=& \lim_{n\to\infty}\tr(a_0\varphi_{1\cap\infty}(h)^{1/2}b_n\varphi_{1\cap\infty}(h)^{1/2}a_1^*)\\
&=& \tr(a_0\varphi_{1\cap\infty}(h)^{1/2}b\varphi_{1\cap\infty}(h)^{1/2}a_1^*).
\end{eqnarray*}
From the above it is clear that the functional $F$ is canonically induced by $b$. Having established this correspondence, a modification of the argument of Theorem \ref{con2} now suffices to prove that $b$ is of the form $b = \iota_1(b_1) + \iota_{2,r}(b_{2,r}) + \iota_{2,l}(b_{2,l}) + \iota_\infty(b_\infty)$, where $b_1\in L^1(\M)$, $b_{2,r}\in L^2(\M;r)$, $b_{2,l}\in L^2(\M;l)$ and $b_\infty\in L^\infty(\M)$. Since the functionals $F_{b_n}$ converge in norm to a functional induced by an element of $X_{1+\infty}$, the sequence $\{b_n\}$ converges to $b$ with respect to the norm $\|\cdot\|_{1+\infty}$.
\end{proof}

It is the author's suspicion that in fact $X_{1+\infty}$ agrees with $L_{1+\infty}(\M)$. We present the following Proposition in support of this contention. Should this be the case the map $b\to b$ will then be a continuous map from the Banach space $(L_{1+\infty}(\M), \|\cdot\|_{1+\infty})$ to the quasi-Banach space $(L_{1+\infty}(\M), \mu_1(\cdot))$. An application of the Open Mapping Theorem will then suffice to show that these two topologies are in fact homeomorphic, and hence that $L_{1+\infty}(\M)$ is normable. 

\begin{proposition}
Let $\M$ be a semifinite von Neumann algebra and let the weight $\nu =\tau_\nu$ be an fns trace on $\M$. Then  $X_{1+\infty}$ agrees with $L_{1+\infty}(\M)$.
\end{proposition}

\begin{proof}
It suffices to show that $L_{1+\infty}(\M) \subset X_{1+\infty}$. From the previous section we know that in this setting $L_{1+\infty}(\M)$ is canonically isomorphic to $L_{1+\infty}(\M,\tau_\nu)$. From the preliminary material in the introductory section we know that each element $a \in L_{1+\infty}(\M,\tau_\nu) \subset \tM$ must satisfy the requirement that $\psi_{1+\infty}(\alpha |a|) \in L^1(\M,\tau_\nu)$ for some $\alpha > 0$, where $\psi_{1+\infty}$ is the Orlicz function defined as in equation \ref{eq:1inf}. For the sake of argument suppose that $\alpha=1$. Observe that then $\psi_{1+\infty}(|a|) = (|a|-\I)\chi_{[1,\infty)}(|a|)$. So $|a|$ may be written as the sum of the bounded operator $|a|\chi_{[0,1]}(|a|) + \chi_{[1,\infty)}(|a|) \in \M$ and the element $(|a|-\I)\chi_{[1,\infty)}(|a|)$ of $L^1(\M,\tau_\nu)$. If we let $v$ be the partial isometry in the polar decomposition of $a$, we have that $a = [a\chi_{[0,1]}(|a|) + v\chi_{[1,\infty)}(|a|)] + (a-v)\chi_{[1,\infty)}(|a|) \in L^1(\M,\tau_\nu) + L^\infty(\M,\tau_\nu)$. To prove the claim we simply transfer this information to the space $L_{1+\infty}(\M)$ by means of the canonical map from this space to $L_{1+\infty}(\M,\tau_\nu)$.
\end{proof}

We now return to the promised verification of the equivalence of the conditions listed in Theorem \ref{K3gen}.

\bigskip 

\noindent\textbf{Theorem \ref{K3gen}${}^\prime$} \emph{The following are equivalent:
\begin{enumerate}
\item $a\in S^{\psi}$.
\item $S_{\psi^*} a^* \subset L^2(\M)$.
\item $a\widetilde{\varphi}_{\psi^*}(h)^{1/2}e \in L^2(\M)$ for all projections $e\in \mathfrak{n}_\nu$.
\end{enumerate}}

\bigskip

\begin{proof}
We need only prove that (3) $\Rightarrow$ (1). Hence let $a, a_0 \in \widetilde{\mathcal{A}}$ be given so that $a\widetilde{\varphi}_{\psi^*}(h)^{1/2}e, a_0\widetilde{\varphi}_{\psi^*}(h)^{1/2}e  \in L^2(\M)$ for all projections $e\in \mathfrak{n}_\nu$. On 
arguing as in Proposition \ref{orldef}, we can show that then $$a[\widetilde{\varphi}_{\psi^*}(h)^{1/2}b], a_0[\widetilde{\varphi}_{\psi^*}(h)^{1/2}b]  \in L^2(\M)\mbox{ for all } b\in \mathfrak{n}_\nu.$$ Since (as was noted at the start of this subsection) 
$\zeta_\psi(t) = \frac{\widetilde{\varphi}_{\psi^*}(t)}{\varphi_{1\cap\infty}(t)}$ is bounded, $\zeta_\psi(h)\in\mathcal{A}$. It is then an exercise 
to see that $\zeta_\psi(h)^{1/2}a_0^*a\zeta_\psi(h)^{1/2}\in L_{1+\infty}(\M)$ and that the above fomula may therefore be written in the form 
$$a\zeta_\psi(h)^{1/2}[\varphi_{1\cap\infty}(h)^{1/2}b], a_0\zeta_\psi(h)^{1/2}[\varphi_{1\cap\infty}(h)^{1/2}b]  \in L^2(\M)\mbox{ for all } b\in \mathfrak{n}_\nu.$$However as can be seen from the first part of the proof of Theorem \ref{con2}, this is precisely what we need to ensure that $\zeta_\psi(h)^{1/2}a_0^*a\zeta_\psi(h)^{1/2}$ induces a well-defined bounded linear functional on 
$\varphi_{1\cap\infty}(h)^{1/2}\mathfrak{m}_\nu\varphi_{1\cap\infty}(h)^{1/2}$. So in fact 
$$\zeta_\psi(h)^{1/2}a_0^*a\zeta_\psi(h)^{1/2}\in X_{1+\infty} = \iota_1(L^1(\M)) + \iota_{2,r}(L^2(\M;r)) + \iota_{2,l}(L^2(\M;l)) + \iota_\infty(L^\infty(\M)).$$Let $\chi_{[0,\delta]}=\chi_{[0,\delta]}(h)$ be the spectral projections from the spectral resolution of $h$, and let  $g\in L^\infty(\M)$ be given. On making use of the facts that $\theta_s(\chi_{[0,\delta]}(h)) = \chi_{[0,\delta]}(\theta_s(h))=\chi_{[0,\delta]}(e^{-s}h)=\chi_{[0,e^s\delta]}(h)$ and that $\varphi_{1\cap\infty}(h).\varphi_{1+\infty}(h)=h$, we then have that 
\begin{eqnarray*}
&& \theta_s((\chi_{[0,\delta_0]}.\varphi_{1\cap\infty}(h)^{1/2})\iota_\infty(g)(\varphi_{1\cap\infty}(h)^{1/2}.\chi_{[0,\delta]}))\\
&=& \theta_s((\chi_{[0,\delta_0]}h^{1/2})g(h^{1/2}\chi_{[0,\delta]}))\\
&=& e^{-s}(\chi_{[0,e^s\delta_0]}h^{1/2})g(h^{1/2}\chi_{[0,e^s\delta]})\\
&=& e^{-s}(\chi_{[0,e^s\delta_0]}\varphi_{1\cap\infty}(h)^{1/2})\iota_\infty(g)(\varphi_{1\cap\infty}(h)^{1/2}\chi_{[0,e^s\delta]}).
\end{eqnarray*}
Similar statements hold for elements of each of $\iota_1(L^1(\M))$, $\iota_{2,r}(L^2(\M;r))$, and $\iota_{2,l}(L^2(\M;l))$. Applying these observations to $\zeta_\psi(h)^{1/2}a_0^*a\zeta_\psi(h)^{1/2}$, therefore yields the conclusion that 
\begin{eqnarray}\label{eq:ab}
&& \theta_s((\chi_{[0,\delta_0]}.\widetilde{\varphi}_{\psi^*}(h)^{1/2})a_0^*a(\widetilde{\varphi}_{\psi^*}(h)^{1/2}.\chi_{[0,\delta]}))\\
&=& \theta_s((\chi_{[0,\delta_0]}.\varphi_{1\cap\infty}(h)^{1/2})\zeta_\psi(h)^{1/2}a_0^*a\zeta_\psi(h)^{1/2}(\varphi_{1\cap\infty}(h)^{1/2}.\chi_{[0,\delta]}))\nonumber\\
&=& e^{-s}(\chi_{[0,e^s\delta_0]}.\varphi_{1\cap\infty}(h)^{1/2})\zeta_\psi(h)^{1/2}a_0^*a\zeta_\psi(h)^{1/2}(\varphi_{1\cap\infty}(h)^{1/2}.\chi_{[0,e^s\delta]})\nonumber\\
&=& (\chi_{[0,e^s\delta_0]}.\widetilde{\varphi}_{\psi^*}(h)^{1/2})a_0^*a(\widetilde{\varphi}_{\psi^*}(h)^{1/2}.\chi_{[0,e^s\delta]})\nonumber
\end{eqnarray}
for any $\delta, \delta_0 >0$. Next notice that for any $b\in \mathfrak{n}_\nu$, Proposition \ref{prop:meas} ensures that $[b\varphi_\psi(h)^{1/2}]$ is a well-defined element of $\widetilde{\mathcal{A}}$ for which we have that $$[b\varphi_\psi(h)^{1/2}][\widetilde{\varphi}_{\psi^*}(h)^{1/2}e] = [bh^{1/2}]e \in L^2(\M)\mbox{ for all projections }e\in \mathfrak{n}_\nu.$$We may therefore set $a_0=[b\varphi_\psi(h)^{1/2}]$ where $b\in \mathfrak{n}_\nu$. In that case equation \ref{eq:ab} reduces to the claim that 
\begin{eqnarray*}
e^{-s/2}(\chi_{[0,e^s\delta_0]}[h^{1/2}b^*])\theta_s(a(\widetilde{\varphi}_{\psi^*}(h)^{1/2}.\chi_{[0,\delta]}))&=&
\theta_s((\chi_{[0,\delta_0]}[h^{1/2}b^*])a(\widetilde{\varphi}_{\psi^*}(h)^{1/2}.\chi_{[0,\delta]}))\\
&=& (\chi_{[0,e^s\delta_0]}[h^{1/2}b^*])a(\widetilde{\varphi}_{\psi^*}(h)^{1/2}.\chi_{[0,e^s\delta]})
\end{eqnarray*}
for any $\delta, \delta_0 >0$. Given $\alpha > 0$, left multiply by the bounded operator $\chi_{[\alpha, \infty]}h^{-1/2}$ to get 
$$\chi_{[\alpha,e^s\delta_0]}b^*\theta_s(a(\widetilde{\varphi}_{\psi^*}(h)^{1/2}.\chi_{[0,\delta]}))= e^{-s/2}\chi_{[\alpha,e^s\delta_0]}b^*a(\widetilde{\varphi}_{\psi^*}(h)^{1/2}.\chi_{[0,e^s\delta]}).$$If we let $\alpha$ decrease to 0 and $\delta_0$ 
increase without bound, this in turn yields the fact that $$b^*\theta_s(a(\widetilde{\varphi}_{\psi^*}(h)^{1/2}.\chi_{[0,\delta]}))= e^{-s/2}b^*a(\widetilde{\varphi}_{\psi^*}(h)^{1/2}.\chi_{[0,e^s\delta]})\mbox{ for any }b\in\mathfrak{n}_\nu.$$Finally recall that there exists an 
increasing approximate unit $\{b_\alpha\}$ in $(\mathfrak{n}_\nu)^+$. (See exercise 2.4.8(d) of \cite{S}.) Using this approximate unit, we therefore arrive at the required conclusion that
\begin{eqnarray*}
\theta_s(a(\widetilde{\varphi}_{\psi^*}(h)^{1/2}.\chi_{[0,\delta]}))&=&\lim_\alpha b_\alpha\theta_s(a(\widetilde{\varphi}_{\psi^*}(h)^{1/2}.\chi_{[0,\delta]}))\\ 
&=& e^{-s/2}\lim_\alpha b_\alpha a(\widetilde{\varphi}_{\psi^*}(h)^{1/2}.\chi_{[0,e^s\delta]})\\
&=& e^{-s/2}a(\widetilde{\varphi}_{\psi^*}(h)^{1/2}.\chi_{[0,e^s\delta]})
\end{eqnarray*} 
for any $\delta > 0$. 
\end{proof}

\subsection{A context for interpolation}

We close this section by indicating how the theory of type III Orlicz spaces may be used to construct type III analogues of noncommutative Banach Function Spaces. For this we need some rudimentary facts regarding the $K$-method of interpolation. These brief comments on the $K$-method of interpolation are extracted from the book of Bennett and 
Sharpley \cite{BS}. Readers wishing to have more extensive details than those provided hereafter are referred to this book. A pair of Banach spaces $X_0$ and $X_1$ are called a compatible couple (or Banach couple) if there exists a Hausdorff topological vector space $\mathfrak{X}$ into which each of $X_0$ and $X_1$ are continuously embedded. If we 
regard these Banach spaces as subspaces of $\mathfrak{X}$, we may give meaning to the notion of the intersection $X_0\cap X_1$ and sum $X_0+X_1$ of such a compatible couple. These new spaces turn out to be Banach spaces themselves 
when equipped with the norms $$\|f\|_{X_0\cap X_1} = \max\{\|f\|_{X_0}, \|f\|_{X_1}\},$$and $$\|f\|_{X_0+X_1} = \inf\{\|f_0\|_{X_0}+\|f_1\|_{X_1}: f = f_0+f_1\}.$$ Given such a compatible couple, for each $f \in X_0+X_1$ and each $t>0$, we define the associated $K$-functional to be the quantity
$$K(f, t; X_0, X_1) = \inf\{\|f_0\|_{X_0}+t\|f_1\|_{X_1}: f = f_0+f_1\}$$where the infimum extends over all representations $f = f_0+f_1$ of $f$ with $f_0\in X_0$ and $f_1\in X_1$. For each fixed $f \in X_0+X_1$, the $K$-functional turns out to be a non-negative concave function of $t\geq 0$, and hence may be written in the form 
$$K(f, t; X_0, X_1) = K(f, 0+; X_0, X_1) + \int_0^t k(f, s; X_0, X_1)\, ds,$$where for each $f$ the so-called 
$k$-functional $k(f, s; X_0, X_1)$ is a uniquely defined non-negative, decreasing, right-continuous function of $s>0$. In cases where the spaces $X_0$ and $X_1$ are fixed, we simply write $K(f,t)$ and $k(f,s)$ for these 
functionals. For each so-called rearrangement invariant monotone Riesz-Fischer norm $\rho$ on the cone of non-negative locally finite measurable functions on $([0,\infty), m)$, the space $(X_0,X_1)_\rho$ consisting of all $f \in \overline{(X_0\cap X_1)}^{X_0} + X_1$ for which $\rho(k(f,\cdot))<\infty$, turns out to be a Banach space when equipped with the norm $\|f\|_\rho = \rho(k(f,\cdot))$. (Refer to the proof of 
\cite[Theorem 5.1.19]{BS} for this verification.) 

[A monotone Riesz-Fischer norm $\rho$ of the type mentioned above is just a functional on the cone of non-negative locally finite measurable functions on $([0,\infty), m)$, which satisfies the properties of a Banach Function norm (as mentioned in the preliminaries) on this cone, as well as the the additional restrictions that 
\begin{itemize}
\item $\rho(g)<\infty$ whenever $\rho(f) <\infty$ and $\int_0^t g(s)\, ds \leq \int_0^t f(s)\, ds$ for all $t>0$;
\item whenever $m(E) < \infty$ for some measurable set, we have that $\rho(\chi_E) < \infty$ and that there exists $C_E > 0$ so that $\int_E f \, dm \leq C_E\rho(f)$ for all $f\in L^0_+(0,\infty)$;
\item for any sequence of non-negative locally finite measurable functions $\{f_n\}$ we have that $\rho(\sum_{n=1}^\infty f_n) \leq \sum_{n=1}^\infty \rho(f_n)$.
\end{itemize}
This includes a very wide class of Banach Function norms.]

For a resonant measure space $(X, \Sigma, \nu)$ it is known that the exact interpolation spaces between $L^1(X,\Sigma,\nu)$ and $L^\infty(X,\Sigma,\nu)$ are just the Banach Function spaces with rearrangement invariant monotone Riesz-Fischer norms. In the type III setting we may therefore formally define type III analogues of such spaces to be appropriate ``intermediate spaces'' of $L^1(\M)$ and $L^\infty(\M)$. However the scale of spaces that we will end up with, will depend on how we give meaning to the idea of sum and intersection of $L^1(\M)$ and $L^\infty(\M)$. As subspaces of $\widetilde{\mathcal{A}}$ we of course have that $L^1(\M)\cap L^\infty(\M) = \{0\}$ in a na\"ive set theoretic sense. This is however far too simplistic an approach. Instead we use the language of Orlicz spaces and propose a scale where the role of the intersection and sum of $L^1(\M)$ and $L^\infty(\M)$ will be played by $L^{1\cap\infty}$ and $X_{1+\infty}$ respectively. From our study of these spaces we see that this approach has the advantage that it takes the $L^2$ norm into account at the very outset of the theory. Direct interpolation between $L^{1\cap\infty}$ and $X_{1+\infty}$ by means of the $K$-method is unlikely to yield the correct scale of spaces. (See exercise 1(d) on page 426 of \cite{BS}.) What we need to do is modify the $K$-method to obtain a framework where the spaces $L^{1\cap\infty}$ and $X_{1+\infty}$ naturally play the role of intersection and sum when interpolating between $L^1(\M)$ and $L^\infty(\M)$. The role that is played by the decreasing rearrangement in the semifinite setting, will then be played by the $k$-functional in this modified $K$-method. In support of this proposal we note the following fact:

\begin{proposition}
Let $\M$ be a semifinite von Neumann algebra and the weight $\nu =\tau_\nu$ an fns trace on $\M$. The $k$-functional for the compatible couple $X_0=L^{1}(\M,\tau)$ and $X_1=L^{\infty}(\M,\tau)$ is precisely $k(f, s) = \mu_s(f)$.
\end{proposition}

\begin{proof}
This is a direct consequence of the formula proved directly after Theorem 4.4 of \cite{FK}. Specifically of the fact that $$\int_0^t\mu_s(f)\, ds = \inf\{\|f_0\|_{1}+t\|f_1\|_{\infty}: f = f_0+f_1\}$$ for each $f \in L^1(\M,\tau)+L^\infty(\M,\tau)$, where the infimum extends over all representations $f = f_0+f_1$ of $f$ with $f_0\in L^{1}(\M,\tau)$ and $f_1\in L^{\infty}(\M,\tau)$.
\end{proof}
    
\begin{definition}[Interpolation between $L^1(\M)$ and $L^\infty(\M)$ -- modified $K$-method]
For each $f\in X_{1+\infty}$ and each $t>0$, we define the associated $\widetilde{K}$-functional to be the quantity
$$\widetilde{K}(f, t) = \inf(\|f_1\|_{1}+\sqrt{t}\|f_{2,r}\|_2 + \sqrt{t}\|f_{2,l}\|_2 +t\|f_\infty\|_{\infty})$$where the infimum extends over all representations $f = \iota_1(f_1) + \iota_{2,r}(f_{2,r}) + \iota_{2,l}(f_{2,l}) + \iota_\infty(f_\infty)$ of $f$ with $f_1\in L^1(\M)$, $f_{2,r}\in L^2(\M;r)$, $f_{2,l}\in L^2(\M;l)$ and $f_\infty\in L^\infty(\M)$.
\end{definition}

\begin{proposition}\label{416} 
For each fixed $f \in X_{1+\infty}$, the $\widetilde{K}$-functional is a a non-negative non-decreasing concave function of $t\geq 0$, and hence may be written in the form 
$$\widetilde{K}(f, t) = \widetilde{K}(f, 0+) + \int_0^t \widetilde{k}(f, s)\, ds,$$where the functional $\widetilde{k}(f, s)$ is a uniquely defined non-negative, decreasing, right-continuous function of $s>0$.
\end{proposition}

\begin{proof}
Once we prove that $\widetilde{K}(f, t)$ is non-increasing and concave, the integral representation of this functional mentioned in the hypothesis, will follow directly from classical theory. The verification that $\widetilde{K}(f, t)$ is non-increasing is a fairly straightforward exercise, and hence we only prove the fact that it is concave. With this in mind let $f = \iota_1(f_1) + \iota_{2,r}(f_{2,r}) + \iota_{2,l}(f_{2,l}) + \iota_\infty(f_\infty)$ be a representation of $f\in X_{1+\infty}$ with $f_1\in L^1(\M)$, $f_{2,r}\in L^2(\M;r)$, $f_{2,l}\in L^2(\M;l)$ and $f_\infty\in L^\infty(\M)$, and let $t = \alpha t_1 + (1-\alpha)t_2$ where $0\leq \alpha\leq 1$, and $t_1, t_2 >0$. Notice that $[0,\infty) \to [0,\infty): t \mapsto \sqrt{t}$ is itself a concave function in that the line segment joining $(t_1, \sqrt{t_1})$ and $(t_2, \sqrt{t_2})$ lies below the graph of this function. This in turn ensures that $$\alpha\sqrt{t_1} + (1-\alpha)\sqrt{t_2} \leq \sqrt{\alpha t_1 + (1-\alpha)t_2} = \sqrt{t}.$$Using this inequality it is now clear that
\begin{eqnarray*}
\alpha\widetilde{K}(f, t_1) + (1-\alpha)\widetilde{K}(f, t_2) &\leq& \alpha(\|f_1\|_{1}+\sqrt{t_1}\|f_{2,r}\|_2 + \sqrt{t_1}\|f_{2,l}\|_2 +t_1\|f_\infty\|_{\infty})\\
& & + (1-\alpha)(\|f_1\|_{1}+\sqrt{t_2}\|f_{2,r}\|_2 + \sqrt{t_2}\|f_{2,l}\|_2 +t_2\|f_\infty\|_{\infty})\\
&=& \|f_1\|_1 + (\alpha\sqrt{t_1} + (1-\alpha)\sqrt{t_2})(\|f_{2,r}\|_2 \|f_{2,l}\|_2) + t\|f_\infty\|_\infty\\
&\leq& \|f_1\|_1 + \sqrt{t}(\|f_{2,r}\|_2 + \|f_{2,l}\|_2) + t\|f_\infty\|_\infty.
\end{eqnarray*}
Taking the infimum over all such representations of $f$ then yields the fact that $$\alpha\widetilde{K}(f, t_1) + (1-\alpha)\widetilde{K}(f, t_2) \leq \widetilde{K}(f, t),$$as required.  
\end{proof}

In passing from the semifinite setting to the type III setting, the role played by $\mu_t(f)$ in the former, could therefore reasonably be played by $\widetilde{k}(f, t)$ in the latter. Given any monotone Riesz-Fischer norm $\rho$, we may therefore propose to define the associated non-commutative Riesz-Fischer space $L_{RF}^\rho(\M)$ of $\M$ to be the space of all $f \in X_{1+\infty}$ for which $\rho(\widetilde{k}(f,\cdot)) < \infty$, and propose the quantity $\|f\|_\rho = \rho(\widetilde{k}(f,\cdot))$ as a potential norm.

\section{Orlicz spaces for $\sigma$-finite algebras}

Throughout this section we will assume that $\M$ is a $\sigma$-finite von Neumann algebra equipped with a faithful normal state $\nu$. We point out that this ensures that density $h = \frac{d\widetilde{\nu}}{d\tau}$ is is in fact $\tau_{\mathcal{A}}$-measurable. This fact will be used throughout this section, and not only helps to greatly simplify the theory, but also opens the way for the introduction of ``left'' and ``right'' Orlicz spaces in this context. 

\begin{definition}
Let $\psi$ be an Orlicz function. We define the left Orlicz spaces corresponding to the Luxemburg and Orlicz norms to respectively be
$$L^\psi(\M;l) =\{a\in \widetilde{\mathcal{A}}: \widetilde{\varphi}_{\psi^*}(h)a \in L^1(\M)\}$$and $$L_\psi(\M;l) =\{a\in \widetilde{\mathcal{A}}: \varphi_{\psi^*}(h)a \in L^1(\M)\}.$$ The right Orlicz spaces $L^\psi(\M;r)$ and $L_\psi(\M;r)$ are defined similarly.
\end{definition} 

We proceed with the task of describing the topologies on these spaces. As may be expected the theory closely parallels the general theory presented in the previous section. 

\begin{theorem}
Let $\psi$ be an Orlicz function. For any $a$ in either $L^\psi(\M;l)$ or $L_\psi(\M;l)$, we then have that $$t\mu_t(a) \leq \mu_1(a) \qquad\mbox{for all} \qquad 0<t\leq 1.$$ 
\end{theorem}
 
\begin{proof}
It suffices to prove the theorem for space $L^\psi(\M;l)$. So let $a \in L^\psi(\M;l)$, and $0<t\leq1$ be given, and write $t$ as $t=e^s$ where $s\leq 0$. The proof now runs along similar lines as the one for symmetric Orlicz spaces. The primary difference is that here we may conclude that $$\theta_s(a) = \frac{1}{t}b_ta$$where $b_t = \widetilde{\varphi}_{\psi^*}(\tfrac{1}{t}h)^{-1}\widetilde{\varphi}_{\psi^*}(h)$. To see this note that by the definition of $L^1(\M)$ and $L^\psi(\M;l)$ we have that 
\begin{eqnarray*}
e^{-s}(\widetilde{\varphi}_{\psi^*}(h)a) &=& \theta_s(\widetilde{\varphi}_{\psi^*}(h)a)\\
&=& (\widetilde{\varphi}_{\psi^*}(\theta_s(h))\theta_s(a)\\
&=& \widetilde{\varphi}_{\psi^*}(e^{-s}h)\theta_s(a).
\end{eqnarray*}
As before this fact combined with the observation that $|b_t| \leq \I$, leads to the required conclusion.
\end{proof} 
 
\begin{proposition}\label{sigmaqntop}
The quantity $\mu_1(\cdot)$ is a quasinorm for both $L^\psi(\M;l)$ and $L_\psi(\M;l)$. The topology induced on these spaces by this quasinorm is complete and is homeomorphic to the topology of convergence in measure inherited from $\widetilde{\mathcal{A}}$. 
\end{proposition} 

\begin{proof}
We consider only the space $L^\psi(\M;l)$.

First suppose that we are given $a \in L^\psi(\M;l)$ with $\mu_1(a)=0$. We have that 
$$\|\varphi_{\psi^*}(h)a\|_1 = 2\mu_2(\varphi_{\psi^*}(h)a) \leq 2\mu_1(\varphi_{\psi^*}(h))\mu_1(a)=0.$$But then $\varphi_{\psi^*}(h)a=0$, which ensures that $a=0$. The verification that $\mu_1$ satisfies a generalised triangle inequality, runs along similar lines as before. In this case the fact that $L^\psi(\M;l)$ is a closed subspace of $\widetilde{\mathcal{A}}$, follows from the observation that $L^\psi(\M;l)$ is nothing but the intersection of the kernels of the operators $$a \to \theta_s(\widetilde{\varphi}_{\psi^*}(h)a) - e^{-s}\widetilde{\varphi}_{\psi^*}(h)a \qquad s\in \mathbb{R},$$all of which are continuous in the topology of convergence in measure. The proof that the topology induced on $L^\psi(\M;l)$ by $\mu_1$ is precisely the topology of convergence in measure, is almost entirely analogous to that of the former case.
\end{proof}

\begin{theorem}
Let $\psi$ be an Orlicz function. The topology of convergence in measure on $L^\psi(\M;l)$ (respectively 
$L_\psi(\M;l)$) is normable whenever the upper Boyd index of $L^\psi(\mathbb{R})$ (respectively $L_\psi(\mathbb{R})$) is strictly less than 1. 
\end{theorem}

\begin{proof} The proof of Theorem \ref{ntop} easily adapts.
\end{proof}

Our next result describes the content of K\"othe duality in this context.

\begin{proposition}\label{K3fin} For the spaces defined above we have that 
$$L_{\psi^*}(\M;r) = \{a\in \widetilde{\mathcal{A}}: ab \in L^1(\M), \mbox{for all } b\in L^\psi(\M;l)\}$$and 
$$L^{\psi^*}(\M;r) = \{a\in \widetilde{\mathcal{A}}: ab \in L^1(\M), \mbox{for all } b\in L_\psi(\M;l)\}.$$
\end{proposition}

\begin{proof}
We only prove the first equality. Notice that $\widetilde{\varphi}_{\psi}(h)\varphi_{\psi^*}(h)=h$. This ensures that $\varphi_{\psi^*}(h) \in  L^{\psi^*}(\M;l)$, which in turn establishes the inclusion ``$\supset$''. To see the converse inclusion notice that for any $a \in L_{\psi^*}(\M;r)$ and $b \in L^{\psi}(\M;l)$ the fact that $a\widetilde{\varphi}_{\psi}(h), \varphi_{\psi}(h)b \in L^1(\M)$, ensures that $$\theta_s(ab) = \theta_s(a\widetilde{\varphi}_{\psi}(h).h^{-1}.\varphi_{\psi}(h)b) = e^{-s}a\widetilde{\varphi}_{\psi}(h).e^{s}h^{-1}.e^{-s}\varphi_{\psi}(h)b =e^{-s}ab$$for all $s\in \mathbb{R}$. Hence $ab \in L^1(\M)$. 
\end{proof}

\begin{corollary}\label{dualaction}
The spaces $L^\psi(\M;l)$, $L_{\psi^*}(\M;r)$ form a dual pair with the dual action of the spaces on each other defined by $$\langle a,b\rangle = \tr(ba) \quad\mbox{for all } \quad a\in L^\psi(\M;l), b\in L_{\psi^*}(\M;r).$$(A similar conclusion holds for the pair $L_\psi(\M;l)$, $L^{\psi^*}(\M;r)$. 
\end{corollary}

\begin{proof}
Given any $a\in L^\psi(\M;l)$, $b\in L_{\psi^*}(\M;r)$ it follows from the proposition that $ba \in L^1(\M)$. We then have that $|\langle a,b\rangle| \leq \tr(|ba|) = 2\mu_2(ba) \leq 2\mu_1(b)\mu_1(a)$. 
\end{proof}

\section{Open questions}

We will focus our discussion on the space $L^\psi(\M)$. Similar comments of course hold for $L_\psi(\M)$. The one outstanding question to be settled in the case of semifinite algebras equipped with an \emph{fns} trace, is whether for any $a\in L_\psi(\M,\tau_\M)$ the norm $\|a\|_\psi^O$ is equal to (rather than merely equivalent to) the quantity $\mu_1(a\otimes\widetilde{\varphi}_{\psi^*}(e^t))$.

In the general case the most important issues that still need to be fully settled are the questions regarding normability and duality for the spaces $L^\psi(\M)$. These two questions are of course interrelated. Normability of a large class of Orlicz spaces has been proven, but can we do better? An obvious candidate for a norm on $L^\psi(\M)$ would be the quantity $$|\!|\!|a|\!|\!|_\psi=\sup\left\{\frac{|\tr(b_0ab_1^*)|}{\mu_1(b_0)\mu_1(b_1)} : b_0, b_1\in S_{\psi^*}\right\}.$$Given $a\in L^\psi(\M)$ and $b_0, b_1\in S_{\psi^*}$, it readily follows that $|\tr(b_0ab_1^*)|=3\mu_3(b_0ab_1^*)\leq 3\mu_1(b_0)\mu_1(b_1)\mu_1(a)$, and hence that $|\!|\!|a|\!|\!|_\psi\leq \mu_1(a)$. But are these 
quantities in fact equivalent on $L^\psi(\M)$? A related question is to determine the size of the space $\mathrm{span}[(S_{\psi^*})^*S_{\psi^*}]$ inside $L_{\psi}(\M)$. Finally in a situation where indeed $L^{\psi}(\M) = \mathrm{span}[(S^\psi)^*S^\psi]$ and 
$L_{\psi^*}(\M) = \mathrm{span}[(S_{\psi^*})^*S_{\psi^*}]$, do we then actually have that $$\wtr(ab)\leq K\mu_1(a)\mu_1(b)$$for some $K>0$ and all $a\in L^\psi(\M)$ and $b\in L_{\psi^*}(\M)$? 

\end{document}